\documentclass[11pt,a4paper,twoside,reqno]{amsart}
\pdfoutput=1
\title{New examples of $G_2$-instantons on $\R^4 \times S^3$} 
\author{Izar Alonso}
\address[I. Alonso]{Department of Mathematics \\ Rutgers University \\Piscataway, NJ 08854\\ USA}
\email{izar.alonso@rutgers.edu}
\thanks{The author was supported by the EPSRC (Grant No.\ 2271784) and the National Science Foundation under Grant No.\ DMS-1928930}

\usepackage[dutch, english]{babel}
\usepackage{fancyhdr,amsmath,amssymb,amsbsy,amsfonts,latexsym,float,color,amsthm,mathabx, hyperref}
\usepackage{graphicx,cite,enumerate,tikz-cd,cancel,anysize,tensor,wasysym,comment,mathtools,enumitem} 
\usetikzlibrary{babel}

\numberwithin{equation}{section} 

\newcommand{\R}{\mathbb{R}}
\newcommand{\C}{\mathbb{C}}
\newcommand{\tr}{\text{tr}}
\newcommand{\iprod}{\mathbin{\raisebox{\depth}{\scalebox{1}[-2]{$\lnot$}}}}

\newtheorem{theorem}{Theorem}[section]
\theoremstyle{definition}
\newtheorem{defi}[theorem]{Definition}
\newtheorem{lemma}[theorem]{Lemma}
\newtheorem{ex}[theorem]{Example}

\newtheorem{prop}[theorem]{Proposition}
\newtheorem{cor}[theorem]{Corollary}
\newtheorem{remark}[theorem]{Remark}

\numberwithin{equation}{section}

\begin{document}

\begin{abstract}
    We study the existence of $\text{SU}(2)^2$-invariant $G_2$-instantons on $\mathbb{R}^4 \times S^3$ with the coclosed $G_2$-structures found on \cite{Alo22}. 
    We find an explicit 1-parameter family of $\text{SU}(2)^3$-invariant $G_2$-instantons on the trivial bundle on $\mathbb{R}^4 \times S^3$ and study its ``bubbling'' behaviour. We prove the existence a 1-parameter family on the identity bundle.
    We also provide existence results for locally defined $\text{SU}(2)^2$-invariant $G_2$-instantons.
\end{abstract}

\maketitle

\section{Introduction}

A \textit{$G_2$-instanton} is a special kind of connection on a Riemannian 7-manifold.
More concretely, a connection $A$ on a principal bundle over a manifold with a $G_2$-structure $\varphi$ is a $G_2$-instanton if 
$$
F_A \wedge \psi =0,
$$
where $F_A$ is the curvature of $A$ and $\psi= {*}_\varphi \varphi$, ${*}_\varphi$ being the Hodge star operator induced by $\varphi$.
They are analogues of anti-self-dual connections in four-dimensional manifolds, 
and also similar to flat connections over three-manifolds, in the sense that they are critical points of a Chern--Simons type functional, the Yang--Mills energy functional. %\cite{SEW15}
The study of $G_2$-instantons is a research topic of interest for the gauge theory community, due to Donaldson–Thomas’ suggestion \cite{DT98} that it may be possible to use $G_2$-instantons to define invariants for $G_2$-manifolds.
Some details of this idea were later worked out by Donaldson and Segal \cite{DS11}.
Although the naive count of $G_2$-instantons on a compact $G_2$-manifold cannot produce a deformation-invariant, one may hope that by counting $G_2$-instantons together with the Seiberg--Witten monopoles on associative submanifolds the mutual degenerations will solve the issue, while giving a relation between $G_2$-instantons and the Seiberg--Witten monopoles \cite{Haydys17}.
Other reason to be interested in $G_2$-instantons is their appearance in heterotic string theory, as they appear in the \textit{heterotic $G_2$ system} (see for example \cite{dlOLM21}). 
This system involves a pair of connections on some vector bundle and the tangent bundle which are respectively $G_2$-instantons, and whose curvatures are coupled by the \textit{heterotic Bianchi identity}.

The first examples of $G_2$-instantons over compact manifolds were constructed by Walpuski \cite{Walpuski13}, over Joyce’s $G_2$-manifolds constructed using the ``Generalised Kummer construction" (\cite{Joyce96}). %(see \cite{joyce}). 
More recently, Platt \cite{Platt22} constructed $G_2$-instantons on Joyce and Karigiannis's generalization of the previous manifolds \cite{JK21}.
On the compact “twisted connected sums” $G_2$-manifolds (\cite{Kovalev03, CHNP15}),  
Sá Earp, Walpuski and Ménet--Sá Earp--Nordström \cite{SEW15, Walpuski16, MNSE21} gave an abstract construction of $G_2$-instantons as well as examples.
On noncompact complete holonomy $G_2$-manifolds, the first examples of $G_2$-instantons where found on the spinor bundle of $S^3$ with the Bryant--Salamon metric \cite{BS89} by Clarke in \cite{Clarke14}; these examples were later generalized in \cite{Oliveira14, Lotay17}.
Very recently, Stein and Turner \cite{SteinTurner} completed the study of $\text{SU}(2)^3$-invariant $G_2$-instantons over the spinor bundle of $S^3$ \cite{Lotay17} with the Bryant--Salamon metric by constructing a new 1-parameter family of examples.
In \cite{MNT22}, Matthies, Nordström and Turner construct a 1-parameter family of $G_2$-instantons on the the asymptotically conical limit of the $\C_7$ family of $G_2$-holonomy metrics of \cite{FHN21}.

All of the constructions on non-compact manifolds used the fact that the manifolds had a cohomogeneity one structure. 
A Riemannian manifold is of \textit{cohomogeneity one} if there is a compact Lie group $H$ acting by isometries with a generic orbit of codimension one. 
The collection of the $H$ with its isotropy groups $H \supset K \supset K_0$ ($H \supset K_+, K_- \supset K_0$ for a compact manifold) is called the \textit{group diagram}, and we can recover the cohomogeneity one manifold from it.
We refer the reader for example to \cite{Berand1982} for a full description of cohomogeneity one manifolds.
Because of the significance of cohomogeneity one techniques to construct examples in the context of special holonomy (from \cite{BS89} to \cite{FHN21}), and in particular for the study of $G_2$-instantons, it is natural to continue exploiting cohomogeneity one symmetries to produce more examples of $G_2$-instantons.

A manifold admitting a $G_2$-instanton has a positive 3-form $\varphi$ known as a $G_2$-structure, which has to be \textit{coclosed}. More concretely, this means that $d {*}_\varphi \varphi =0$. 
In \cite{Alo22}, the author constructed a large family of coclosed $G_2$-structures over the manifold $\R^4 \times S^3$, which are invariant under the cohomogeneity one action of $\text{SU}(2)^2$. 
In this paper, we will construct and classify $G_2$-instantons for the $G_2$-structures of this family.
This expands the study of $\text{SU}(2)^2$-invariant $G_2$-instantons initiated in \cite{Lotay17}
by considering coclosed but not necessarily torsion-free $G_2$-structures. 
Removing this restriction broadens our search, since $\text{SU}(2)^2$-invariant $G_2$-structures and $G_2$-instantons considered previously had an extra $\text{SU}(2)$ or $\text{U}(1)$ symmetry.

\subsection*{Layout and main results}
We start Section \ref{sec:prelim} by giving a brief review on $G_2$-instantons and homogeneous bundles in sections \ref{sec:G2inssummary} and \ref{sec:hombundles}, respectively.
In Section \ref{sec:Alo22summary} we describe the geometry of $\R^4 \times S^3$ and the coclosed $G_2$-structures found in \cite{Alo22}.

In Section \ref{sec:G2instantons} we present a summary of the setup of the problem of finding $G_2$-instantons on a vector bundle over a cohomogeneity one manifold.
In Section \ref{sec:generaleqns}, we state the general equations for an $\text{SU}(2)^2$-invariant connection to be a $G_2$-instanton. Then, in Section \ref{sec:extension} we present the conditions for such connections to extend smoothly to the singular orbit. We consider the particular case when the structure group of the instanton is abelian in Section \ref{sec:abelian}.
From then on, the paper will focus on instantons with the simplest non-abelian group, $\text{SU}(2)$. There are only two $\text{SU}(2)^2$-homogeneous bundles with structure group $\text{SU}(2)$, which should be considered separately throughout most of the paper: the trivial bundle $P_1$ and the identity bundle $P_\text{id}$.

In Section \ref{sec:SU(2)^3}, we study
$G_2$-instantons on $\R^4 \times S^3$ with a larger symmetry group ($\text{SU}(2)^3$-invariant). 
First, we present the ODEs for $\text{SU}(2)^3$-invariant $G_2$-instantons in \ref{sec:3ODEs}. 
In \ref{sec:SU(2)^3coclosed}, we present our main results.
For every $\text{SU}(2)^3$-invariant coclosed $G_2$-structure on \cite{Alo22}, we prove an existence result of an explicit 1-parameter family of $G_2$-instantons, 
extending smoothly to the singular orbit on the trivial bundle $P_1$. Additionally we have a one parameter family of $G_2$-instantons extending smoothly on the identity bundle $P_{\text{id}}$, one of them also being explicit.
\begin{theorem}
Let $M=\R^4 \times S^3$, with a $\text{SU}(2)^3$-invariant coclosed $G_2$-structure given by $A_1$ and $b_0 >0$ as in \cite[Theorem 4.9]{Alo22}. 
There exists two 1-parameter families of smooth $\text{SU}(2)^3$-invariant $G_2$-instantons with gauge group $\text{SU}(2)$: $\theta^{x_1}$, $x_1 \in [0, \infty)$ extending to $P_1$; and $\theta_{y_0}$, $y_0 \in [-1/b_0, 1/b_0]$ extending to $P_\text{id}$.
\end{theorem}
In \ref{sec:behaviour} we analyze the behaviour of sequences of instantons found, which present a ``bubbling" behaviour, and the relation between all $G_2$-instantons encountered.

In Section \ref{sec:SU(2)^2ins}, we study the most general situation of $\text{SU}(2)^2$-invariant $G_2$-instantons. We derive a system of six ordinary differential equations for a connection to be a $G_2$-instanton and the conditions for it to extend smoothly to a singular orbit with bundle $P_1$ or $P_\text{id}$ in \ref{sec:SU(2)^2ODE}. In \ref{sec:P1} we give an existence result of a 3-parameter family of $G_2$-instantons in a neighbourhood of the singular orbit on the trivial bundle.
In \ref{sec:Pid} we study the existence of $G_2$-instantons in a neighbourhood of the singular orbit on the identity bundle, and find an additional 3-parameter family of $G_2$-instantons.

\subsection*{Acknowledgements} 
I would like to thank my supervisors Andrew Dancer and Jason Lotay for the multiple helpful discussions, and for their constant support and advice. 
I would also like to thank Jakob Stein and Gon\c{c}alo Oliveira for helpful discussions. This material is based upon work supported by the National Science Foundation under Grant No.\ DMS-1928930 while the author was in residence at the Simons Laufer Mathematical Sciences Institute (previously known as MSRI) in Berkeley, California, during the Fall 2022 semester. The author was supported by the EPSRC (Grant No.\ 2271784).

\section{Preliminaries}\label{sec:prelim}

\subsection{$G_2$-instantons}\label{sec:G2inssummary}

A \textit{$G_2$-structure} on a smooth seven-manifold $M$ is a smooth 3-form $\varphi$ on $M$ such that, at every $p \in M$, there exists a linear isomorphism $T_p M \cong \R^7$ with respect to which $\varphi_p \in \Lambda^3(T_p^* M)$ corresponds to the standard associative form $\varphi_0 = e^{123} -e^{167} -e^{527} -e^{563} -e^{415} -e^{426} -e^{437} \in \Lambda^3(\R^7)^*$, where $e^1, ..., e^7$ is the standard dual basis of $(\R^7)^*$.
A $G_2$-structure $\varphi$ induces a Riemannian metric $g_\varphi$ and volume form $\text{vol}_\varphi$ on $M$ such that
$$
(X \iprod \varphi ) \wedge (Y \iprod \varphi ) \wedge \varphi = -6g_\varphi (X,Y) \text{vol}_\varphi.
$$
We define the \textit{coassociative form} as the 4-form $\psi= *_\varphi \varphi$, where $*_\varphi$ is the Hodge star operator induced by $(g_\varphi, \text{vol}_\varphi)$.
We can classify $G_2$-structures in different types: we say that a $G_2$-structure $\varphi$ is \textit{closed} if $d \varphi =0$ and that it is \textit{coclosed} if $d \psi=0$. If it is both closed and coclosed, we say that it is torsion-free, and we call $(M, g_\varphi)$ a $G_2$-manifold. 
For $G_2$-manifolds, the metric is Ricci-flat and the holonomy group contained in the exceptional Lie group $G_2$.
$G_2$-manifolds have been extensively studied, examples of torsion-free $G_2$-structures have been found on non-compact \cite{BS89, Brandhuber01, Bogoyavlenskaya13, FHN21} and compact manifolds \cite{Joyce96, Kovalev03, CHNP15, JK21}.

In this paper, we let $M$ be a seven-dimensional manifold and $\varphi$ be a coclosed $G_2$-structure.
Let $P \rightarrow M$ be a principal bundle with structure group $G$. We assume that $G$, which is also denoted as the \textit{gauge group}, is a compact Lie group. Let $A$ be a connection over the principal bundle $P$.

\begin{defi}
A connection $A$ on $P$ is a \textit{$G_2$-instanton} if 
$$
F_A \wedge \psi=0.
$$
\end{defi}

\begin{remark}
For the definition of $G_2$-instanton, we do need to ask for the structure to be coclosed in order to not get an extra condition when differentiating $F_A \wedge \psi=0$, as the equation would then be over-determined, meaning that $G_2$-instantons would not exist. Most constructions of $G_2$-instantons are on $G_2$-manifolds.
\end{remark}

We can write a decomposition $\Omega^2 = \Omega_7^2 \oplus \Omega_{14}^2$ into irreducible $G_2$ representations of pointwise dimension, where
\begin{equation}
\begin{array}{rl}
    \Omega_7^2&= \{ \beta \in \Omega^2 \ | \ {*} (\varphi \wedge \beta)=-2 \beta \}, \\
    \Omega_{14}^2& = \{ \beta \in  \Omega^2 \ | \ {*} (\varphi \wedge \beta)= \beta \}= \{ \beta \in \Omega^2 \ | \ \beta \wedge \psi=0 \} \cong \mathfrak {g}_2.
\end{array}
\end{equation}
Then, a connection on a $G_2$-manifold is a $G_2$-instanton if and only if 
$$
F_A \wedge \varphi = - * F_A,
$$
or equivalently if $\pi_7(F_A)=0$, where $\pi_7: \Omega^2 \rightarrow \Omega^2_7$ is the orthogonal projection.

On a compact $G_2$-manifold, $G_2$-instantons minimize the Yang--Mills energy functional
$$
\mathcal{YM}(A)=\int_M \vert F_A \vert^2 = \int_M \tr (F_A \wedge * F_A)
$$
on the space of finite-energy connections on $P$.

\subsection{Homogeneous bundles}\label{sec:hombundles}

We will focus on instantons on principal $G$-bundles $P$ that are homogeneous bundles on the homogeneous orbits $H/K$, meaning that the action of $H$ on $H/K$ lifts to a $H$-action on $P$ which commutes with the action of $G$. 
These bundles are determined by their isotropy homomorphism $\lambda : K \rightarrow G$: 
\begin{equation}
   P_\lambda = H \times_{(K,\lambda)} G.
\end{equation}

The Lie algebra $\mathfrak{h}$ of $H$ has a reductive splitting with respect to $K$, which we write as $\mathfrak{h} = \mathfrak{k} \oplus \mathfrak{m}$. 
We call the \textit{canonical invariant connection} to the connection on the bundle $H \rightarrow H/K$ whose horizontal space is $\mathfrak{m}$. Its connection form $A_\lambda^\text{can} \in \Omega^1(H, \mathfrak{g})$ is the left-invariant translation of $d\lambda \oplus 0 : \mathfrak{k} \oplus \mathfrak{m} \rightarrow \mathfrak{g}$. This invariant connection induces a corresponding canonical connection on any $P_\lambda$. 
Wang's theorem classifies other invariant connections, which are in
correspondence with morphisms of $K$-representations:

\begin{theorem}\cite[Theorem 1]{Wang58}
There is a 1-1 correspondence between $G$-invariant connections on $P_\lambda$ and $K$-morphisms:
\begin{equation}
    \Lambda: (\mathfrak{m}, \text{Ad}) \rightarrow (\mathfrak{g}, \text{Ad} \circ \lambda).
\end{equation}
\end{theorem}

We see that any invariant connection will differ from the canonical invariant connection by a morphism $\Lambda$ and the horizontal space of such a connection is given by the kernel of this morphism. This allows us to parameterize all invariant connections on a bundle $P_\lambda$.

\subsection{$\R^4 \times S^3$ and its $\text{SU}(2)^2$-invariant structures}
\label{sec:Alo22summary}

In this text, we focus on the non-compact manifold $\R^4 \times S^3$, which admits a cohomogeneity one action with group diagram $\text{SU}(2) \times \text{SU}(2) \supset \Delta \text{SU}(2) \supset \{ 1 \}$.
In particular, we consider the embedding $\R^4 \times S^3 \hookrightarrow \mathbb{H} \times \mathbb{H}$ and let $\text{SU}(2) \times \text{SU}(2)$ act via
$$
(a_1 , a_2) \cdot (p,q) = (a_1 p, a_1 q \bar{a}_2).
$$
$\R^4 \times S^3$ is diffeomorphic to the total space of the spinor bundle $\mathcal{S} \rightarrow S^3$ over the 3-sphere, which can be described as the quotient 
$$
\dfrac{\text{SU}(2) \times \text{SU}(2) \times \R^4}{\Delta \text{SU}(2)},
$$
where $\text{SU}(2)$ is acting on the right diagonally.
This manifold was first studied by Bryant and Salamon in \cite{BS89}, where the authors used it to construct the first examples of complete $G_2$-manifolds.
Brandhuber et al.\ constructed another $G_2$-holonomy metric (also known as the BGGG metric) on this manifold in \cite{Brandhuber01}. In particular, it is a member of a family of complete $\left( \text{SU}(2)^2 \times \text{U}(1) \right)$-invariant $G_2$-holonomy metrics found by Bogoyavlenskaya in \cite{Bogoyavlenskaya13}.

In \cite{Alo22}, we searched for coclosed $\text{SU}(2)^2$-invariant $G_2$-structures constructed from half-flat $\text{SU}(3)$-structures on this manifold, that are not necessarily torsion-free. 
We found a large family of structures with these conditions on each of those manifolds. 

\begin{theorem}\cite[Theorem 4.9]{Alo22}
On the cohomogeneity one manifold $M=\R^4 \times S^3$ with group diagram $\text{SU}(2)^2 \supset \Delta \text{SU}(2) \supset \{ 1 \}$, there is a family of $\text{SU}(2)^2$-invariant coclosed $G_2$-structures which is given by three positive smooth functions $A_1, A_2, A_3: [0, \infty) \rightarrow \R$
satisfying the boundary conditions at $t=0$
$$
A_i(t)= \dfrac{t}{2} + O(t^3),
$$
and a non-zero parameter.
Moreover, any $\text{SU}(2)^2$-invariant coclosed $G_2$-structure constructed from a half flat $\text{SU}(3)$-structure is in this family.
\end{theorem}

The parameter $b_0$ will be assumed to be positive through this paper for simplicity, as its sign does not have any significant effect.

We construct an $\text{SU}(2)^2$-invariant half-flat $\text{SU}(3)$-structure $(\omega, \Omega_1, \Omega_2)$ on the principal orbits of a cohomogeneity one manifold $M$. 
This follows from the work by Schulte-Hegensbach in his PhD thesis \cite{Schulte10} and Madsen and Salamon in \cite{MS13}.
Then, there is an associated $G_2$-structure on the principal part of the manifold, given by
\begin{equation}\label{eq:G2 structure}
\begin{array}{ll}
\varphi = \text{d}t \wedge \omega(t) + \Omega_1(t), \\
\psi= \dfrac{\omega^2(t)}{2} - \text{d}t \wedge \Omega_2 (t).
\end{array}
\end{equation}

To prove these results, we start by constructing a basis of $\mathfrak{su}(2) \oplus \mathfrak{su}(2)$ written as
$
\mathfrak{su}(2) \oplus \mathfrak{su}(2) = \mathfrak{su}^+(2) \oplus \mathfrak{su}^-(2)
$.
Let $\{ T_i \}_{i=1}^3$ be a basis for $\mathfrak{su}(2)$ such that $[T_i, T_j] = 2 \epsilon_{ijk} T_k$. Then
$$
T_i^+=(T_i, T_i), \quad T_i^-=(T_i, -T_i),
$$
define a basis for $\mathfrak{su}^+(2)$ and $\mathfrak{su}^-(2)$ respectively.
Let $\{ \eta_i^+ \}_{i=1}^3$ and $\{ \eta_i^- \}_{i=1}^3$ be dual basis to $\{ T_i^+ \}_{i=1}^3$ and $\{ T_i^- \}_{i=1}^3$ respectively. 
A generic $\text{SU}(2)^2$-invariant half-flat $\text{SU}(3)$-structure on the principal bundle is given by
six real-valued functions $A_i, B_i: I_t \rightarrow \R$, $i=1,2,3$, where $I_t= (0, \infty)$.
The compatible metric determined by this $\text{SU}(3)$-structure on $\{ t \} \times G/K$ is:
$$
g_t =
\sum_{i=1}^3 (2A_i)^2 \eta_i^+ \otimes \eta_i^+ + \sum_{i=1}^3 (2B_i)^2 \eta_i^- \otimes \eta_i^-,
$$
and the resulting metric on $I_t \times G/K$, compatible with the $G_2$-structure $\varphi = \text{d}t \wedge \omega + \Omega_1$, is given by 
\begin{equation}\label{eq:g}
    g = \text{d}t^2 + g_t.
\end{equation}
Hence, we can see the functions $A_i(t)$ and $B_i(t)$ as describing deformations of the standard cone metric.

\begin{ex}
One particular case of these structures are the torsion-free $G_2$-structures giving the complete Bryant--Salamon $G_2$-holonomy metric on $\R^4 \times S^3$ from \cite{BS89}. 
The Bryant--Salamon metric is actually $\text{SU}(2)^3$-invariant, and it can be realised as a cohomogeneity one manifold with group diagram $\text{SU}(2)^3 \supset \text{SU}(2)^2 \supset \text{SU}(2)$ (where $\text{SU}(2)$ is embedded in $\text{SU}(2)^3$ as $1 \times 1 \times \text{SU}(2)$ and $\text{SU}(2)^2$ as $\Delta_{1,2} \text{SU}(2) \times \text{SU}(2)$), as well as with an action of $\text{SU}(2)^2$ in multiple inequivalent ways. %(singular orbit is $S^3$). 
The extra symmetry means that $A_1 =A_2 =A_3$ and $B_1=B_2=B_3$, and in particular 
$$
A_1= \dfrac{r}{3} \sqrt{1-r^{-3}} , \qquad B_1 =\dfrac{r}{\sqrt{3}},
$$
where $r \in [1, + \infty)$ is a coordinate defined implicitly by 
$
t(r)= \int_1^r \dfrac{ds}{\sqrt{1-s^{-3}}},
$
and $t$ denotes the arc length parameter.
The metric will then be
$$
g=dt^2 + \sum_{i=1}^3  \left(  \dfrac{4 r^2}{9}(1 - r^{-3}) \eta_i^+ \otimes \eta_i^+ + \dfrac{4r^2}{3} \eta_i^- \otimes \eta_i^- \right).
$$
The asymptotic cone over of this metric is the standard homogeneous nearly K\"ahler structure on $S^3 \times S^3$.
\end{ex}

\section{\texorpdfstring{$G_2$}{G2}-instantons on cohomogeneity one manifolds}\label{sec:G2instantons}

\subsection{$G_2$-instanton equations}%General equations
\label{sec:generaleqns}

Let $M= I_t  \times N$ be a seven-dimensional manifold with a $G_2$-structure coming from a half-flat $\text{SU}(3)$-structure.
Let $P$ be a principal $G$-bundle on $M$, then $P$ is a pullback of a bundle on $N$. We will work in temporal gauge, so we can assume that a connection on $P$ is of the form $A=a(t)$, where $a(t)$ is a one-parameter family of connections %on $P$ over $N$. 
on the bundle over $N$. The curvature of $A$ is given by
\begin{equation}\label{eq:FA}
F_A= dt \wedge \dot{a} + F_a(t),
\end{equation}
where $F_a(t)$ is the curvature of the connection $a(t)$. 
Hence, $A$ is a $G_2$-instanton if and only if the following equation for $a(t)$ is satisfied:
\begin{equation}\label{eq:G2instantons}
    \dot{a} \wedge \dfrac{1}{2} \omega^2 -F_a \wedge \Omega_2=0, \quad F_a \wedge \dfrac{1}{2} \omega^2=0.
\end{equation}

\begin{lemma}\cite[Lemma 1]{Lotay17}
Let $M = I_t  \times N$ be equipped with a $G_2$-structure $\varphi$ as in (\ref{eq:G2 structure}) satisfying $\omega \wedge d\omega=0$ and $\omega \wedge \dot{\omega}= -d \Omega_2$, which is equivalent to $d\psi = 0$. Then, $G_2$-instantons $A$ for $\varphi$ are in one-to-one correspondence with one-parameter families of connections $\{ a(t) \}_{t\in I_t}$ solving the equation
\begin{equation}
J_t \dot{a}= - *_t (F_a \wedge \Omega_2),    
\end{equation}
subject to the constraint $\Lambda_t F_a = 0$, where $\Lambda_t$ denotes the metric dual of the operation of wedging with $\omega(t)$. Moreover, this constraint is compatible with the evolution: more precisely, if it holds for some $t_0 \in I_t$, then it holds for all $t \in I_t$.
\end{lemma}

The most general $\text{SU}(2)^2$-invariant connection on any $\text{SU}(2)^2$-homogeneous $G$-bundle 
$P_\lambda =\text{SU}(2)^2 \times_{(K,\lambda)} G$ over a homogeneous space $N=\text{SU}(2)^2/K$, where $\lambda: K \rightarrow G$ is a group homomorphism,
can be written as
\begin{equation}\label{eq:connection}
a=\sum_{i=1}^3 a_i^+ \otimes \eta_i^+ + a_i^- \otimes \eta_i^-,
\end{equation}
where $a_i^\pm \in \mathfrak{g}$ are constant on each principal orbit.

\begin{lemma}\cite[Lemma 2]{Lotay17}\label{lemma:2}
In the previous situation, the curvature of the connection $a(t)$ on $\{ t \} \times N$ is given by
\begin{equation}\label{eq:Fa}
\begin{array}{ll}
    F_a 
    &=\sum_{i=1}^3 [a_i^+, a_i^-] \otimes \eta_i^+ \wedge \eta_i^- \\
    &+\sum_{i=1}^3 ((-2a_i^+ + [a_j^+, a_k^+]) \otimes \eta_j^+ \wedge \eta_k^+ +(-2a_i^+ + [a_j^-, a_k^-]) \otimes \eta_j^- \wedge \eta_k^-)\\
    &+\sum_{i=1}^3 ((-2a_i^- + [a_j^-, a_k^+]) \otimes \eta_j^- \wedge \eta_k^+ +(-2a_i^- + [a_j^+, a_k^-]) \otimes \eta_j^+ \wedge \eta_k^-),
\end{array}
\end{equation}
where in the summation above $\{ j,k \}$ is such that $\{ i,j,k \}$ is a cyclic permutation of $\{ 1,2,3 \}$.
\end{lemma}

\begin{lemma}\cite[Lemma 3]{Lotay17}\label{lemma:3}
Let $\{ i,j,k \}$ be a cyclic permutation of $\{ 1,2,3 \}$.
The equations (\ref{eq:G2instantons}) for $\text{SU}(2)^2$-invariant instantons $a$ on $\R_t^+ \times N$ are
\begin{equation}\label{eq:instantons}
\begin{array}{ll}
\dfrac{B_i}{A_i} \dot{a}_i^+ +\left( \dfrac{B_i}{B_j B_k} - \dfrac{B_i}{A_j A_k} \right) a_i^+ 
= \dfrac{B_i}{2 B_j B_k}[a_j^-, a_k^-] -\dfrac{B_i}{2 A_j A_k}[a_j^+, a_k^+] , \\
\dfrac{A_i}{B_i} \dot{a}_i^- +\left( \dfrac{A_i}{B_j A_k} + \dfrac{A_i}{A_j B_k} \right) a_i^- 
= \dfrac{A_i}{2 B_j A_k}[a_j^-, a_k^+] +\dfrac{A_i}{2 A_j B_k}[a_j^+, a_k^-] ,
\end{array}
\end{equation}
together with the constraint 
\begin{equation}\label{eq:instconstraint}
\sum_{i=1}^3 \dfrac{1}{A_i B_i} [a_i^+, a_i^-]=0.
\end{equation}
\end{lemma}

\subsection{Extension to a singular orbit}\label{sec:extension}

In this section, we present the conditions for the extension of a $\mathfrak{g}$-valued 1-form to the singular orbit $\text{SU}(2)^2 / \Delta \text{SU}(2) \cong S^3$ in $\R^4 \times S^3$.
This conditions are obtained using the Eschenburg-Wang method \cite{EW00}, which for this particular situation was worked out in \cite[Appendix A]{Lotay17}.

We first consider the case when $G=\text{U}(1)$.

\begin{lemma}\cite[Lemma 9]{Lotay17}\label{lemma:9}
The 1-form $b$
$$
b=\sum_{i=1}^3 b_i^+ \otimes \eta_i^+ + \sum_{i=1}^3 b_i^- \otimes \eta_i^-,
$$
extends over the singular orbit $\text{SU}(2)^2 / \Delta \text{SU}(2)$ if and only if the $b_i^\pm$'s are even and $b_i^\pm (0)=0$ for $i = 1, 2, 3$.
\end{lemma}

For most of our analysis, we will take $G= \text{SU}(2)$. The conditions for the extension to the singular orbit will now depend on the choice of bundle. 
Over the principal orbits $H/K_0 \cong \text{SU}(2)^2/ \{ 0 \}$, the only $\text{SU}(2)$-bundle is the trivial one $P= \text{SU}(2)^2 \times \text{SU}(2)$.
% between $P_1$ and $P_\text{id}$.
The singular orbits in the manifold considered is $\text{SU}(2)^2/ \Delta \text{SU}(2)$.
Up to an isomorphism of homogeneous bundles, there are only two possible homomorphisms $\lambda: \text{SU}(2) \rightarrow \text{SU}(2)$, the trivial one and the identity. Hence, there are two choices of bundle:
$$
\begin{array}{ll}
P_1=\text{SU}(2)^2 \times_{(\Delta \text{SU}(2), 1)} \text{SU}(2), \quad
P_\text{id}=\text{SU}(2)^2 \times_{(\Delta \text{SU}(2), \text{id})} \text{SU}(2).
\end{array}
$$
Therefore, we have two possible bundles for $\R^4 \times S^3$, that we will also denote $P_1$ and $P_\text{id}$. 

\begin{lemma}\cite[Lemma 10]{Lotay17}\label{lemma:10}
Let $b$ be an $\mathfrak{su}(2)$-valued 1-form 
$$
b=\sum_{i=1}^3 b_i^+ \otimes \eta_i^+ + \sum_{i=1}^3 b_i^- \otimes \eta_i^-.
$$
Write $b_i^\pm = \sum_{j=1}^3 b_{ij}^\pm T_j$, where $\{ T_i \}_{i=1}^3$ is the standard basis for $\mathfrak{su}(2)$. Then the 1-form $b$ extends over the singular orbit $Q=\text{SU}(2)^2 / \Delta \text{SU}(2)$ if:
\begin{enumerate}[label=(\roman*)]
    \item On the bundle $P_{\text{id}}$: for $i=1,2,3$, $b_{ii}^\pm$'s are even and there are $c_0^-, c_2^\pm \in \R$ such that 
    $$
    b_{ii}^+ = c_2^+ t^2 + O(t^4), \quad b_{ii}^- = c_0^- +c_2^- t^2 + O(t^4)
    $$ 
    and for $i \neq j$, $b_{ij}^\pm = O(t^4)$ are even.
    \item On the bundle $P_1$: $b_{ij}^\pm$'s are even with $b_{ij}^\pm (0)=0$.
\end{enumerate}
\end{lemma}

\subsection{Abelian instantons}\label{sec:abelian}

Suppose the Lie algebra structure of the gauge group is trivial. Then equations (\ref{eq:instantons}) reduce to
\begin{equation}\label{eq:abelianinstantons}
\begin{array}{ll}
\dot{a}_i^+ +\left( \dfrac{A_i}{B_j B_k} - \dfrac{A_i}{A_j A_k} \right) a_i^+ 
= 0 \\
\dot{a}_i^- +\left( \dfrac{B_i}{B_j A_k} - \dfrac{B_i}{A_j B_k} \right) a_i^- 
= 0,
\end{array}
\end{equation}
and we have the following Proposition, which is similar to \cite[Proposition 4]{Lotay17} but for $\R^4 \times S^3$ with a coclosed $G_2$-structure constructed from a half-flat $\text{SU}(3)$-structure. 

\begin{prop}\label{prop:4}
An $\text{SU}(2)^2$-invariant $G_2$-instanton on a $\text{U}(1)$-bundle, or equivalently a complex line bundle on $\R^4 \times S^3$ with a $\text{SU}(2)^2$-invariant coclosed $G_2$-structure as in \cite[Theorem 4.9]{Alo22}, lies in a 3-parameter family. In particular, it can be written as 
\begin{equation}\label{eq:abinsexplicit}
\theta= \sum_{i=1}^3 a_i^+(t_0)t_0^{-2} \exp \left( - \int_{t_0}^t \left( \dfrac{A_i}{B_j B_k} - \dfrac{A_i}{A_j A_k} \right) ds \right) \eta_i^+,
\end{equation}
for some fixed $t_0 \in \R^+$ and $a_i^+(t_0) \in \R$ for $i=1,2,3$ where $\{ i,j,k \}$ is a cyclic permutation of $\{ 1,2,3 \}$.
\end{prop}
\begin{proof}
We observe that the only principal $\text{U}(1)$-bundle is the trivial one, as the only possible isotropy isomorphism $\lambda: \Delta \text{SU}(2) \rightarrow \text{U}(1)$ is trivial.
We compute the coefficients $a_i^\pm$ from (\ref{eq:connection}) integrating in (\ref{eq:abelianinstantons}), and knowing the Taylor expansions around $t=0$ of the term inside the parenthesis, we have
$$
a_i^+(t)= a_i^+(t_0) \exp \left( - \int_{t_0}^t \left( \dfrac{A_i}{B_j B_k} - \dfrac{A_i}{A_j A_k} \right) ds \right) = a_i^+(t_0)t_0^{-2} t^2 +O(t^4),
$$
$$
a_i^-(t)= a_i^-(t_0) \exp \left( - \int_{t_0}^t \left( \dfrac{B_i}{B_j A_k} + \dfrac{B_i}{A_j B_k} \right) ds \right) = a_i^+(t_0)t_0^{4} t^{-4} +O(t^{-2}),
$$
and both of them are even. 
By Lemma \ref{lemma:9}, the corresponding instantons do extend smoothly to the singular orbit at $t=0$ if and only if $a_i^-(t_0)=0$ for $i=1,2,3$. Then, $a_i^+(t_0) \in \R$, $i=1,2,3$ give the 3-parameter family of $G_2$-instantons. 
\end{proof}

If we specialise Proposition \ref{prop:4} to the Bryant Salamon metric, we get the following corollary for $G_2$-instantons with gauge group $\text{U}(1)$.

\begin{cor}\cite[Corollary 1 (a)]{Lotay17}
Any $\text{SU}(2)^2$-invariant $G_2$-instanton $\theta$ with gauge group $\text{U}(1)$ over the Bryant Salamon $G_2$-manifold $\R^4 \times S^3$ can be written as 
$$
\theta= \dfrac{r^3-1}{r} \sum_{i=1}^3 x_i \eta_i^+,
$$
for some $x_1, x_2, x_3 \in \R$, where $r \in [1, +\infty)$ is a coordinate defined implicitly by
$
t(r)= \int_1^r \dfrac{ds}{\sqrt{1-s^{-3}}}.
$
\end{cor}

\section{\texorpdfstring{$\text{SU}(2)^3$}{SU(2)3}-invariant instantons}\label{sec:SU(2)^3}

One of the main technical results used throughout this section is the following theorem, which was originally due to Malgrange \cite{Mal74}, but we refer to Foscolo--Haskins for this statement.

\begin{theorem}\cite[Theorem 4.7]{FH17}\label{thm:4.7}
Consider the singular initial value problem
\begin{equation}\label{eq:4.8}
\dot{y} = \dfrac{1}{t} M_{-1}(y) + M (t, y),
\quad
y(0) = y_0,
\end{equation}
where $y$ takes values in $\R^k$, $M_{-1}: \R^k \rightarrow \R^k$ is a smooth function of $y$ in a neighbourhood of $y_0$ and $M : \R \times \R^k \rightarrow \R^k$ is smooth in $t, y$ in a neighbourhood of $(0, y_0)$. Assume that
\begin{enumerate}[label=(\roman*)]
    \item $M_{-1} (y_0) = 0$;
    \item $h \text{Id} - d_{y_0} M_{-1}$ is invertible for all $h \in \mathbb{N}$, $h \geq 1$.
\end{enumerate}
Then there exists a unique solution $y(t)$ of (\ref{eq:4.8}). Furthermore $y$ depends continuously on $y_0$ satisfying (i) and (ii).
\end{theorem}

We note that throughout our discussion, the parameter $t$ is in $\R_{\geq 0}$ instead of $\R$, but we can still apply the previous theorem by extending the relevant functions to $\R$. To do so, we recall that when we say that a function defined on $\R_{\geq 0}$ is even or odd, we mean that their Taylor expansions at $t=0$ have only even or odd terms respectively, so they could be extended to functions with domain containing $0$ in its interior which are even or odd respectively.

We start our discussion by considering the case when the $G_2$-structure enjoys an extra $\text{SU}(2)$-symmetry, i.e.\ $A_1 =A_2 =A_3$ and $B_1 =B_2 =B_3$. 

\subsection{\texorpdfstring{$\text{SU}(2)^3$}{SU(2)3}-invariant ODEs}\label{sec:3ODEs}

As we already discussed abelian $G_2$-instantons %with gauge group $\text{U}(1)$. 
in Section \ref{sec:abelian}, we now consider the simplest non-abelian gauge group: $\text{SU}(2)$.
The next Proposition simplifies the ODEs and constraints in Lemma \ref{lemma:3} to this case.

\begin{prop}%\cite[Proposition 5]{Lotay17}\label{prop:5}
Let $\theta$ be an $\text{SU}(2)^3$-invariant $G_2$-instanton with gauge group $\text{SU}(2)$ on $\R^4 \times S^3$. There is a standard basis $\{ T_i \}$ of $\mathfrak{su}(2)$ such that (up to an equivariant gauge transformation) we can write
\begin{equation}\label{eq:3.2}
\theta=A_1 x \left( \sum_{i=1}^3 T_i \otimes \eta_i^+ \right) + B_1 y \left( \sum_{i=1}^3 T_i \otimes \eta_i^- \right),
\end{equation}
with $x, y: (0, \infty) \rightarrow \R$ satisfying 
\begin{equation}\label{eq:x}
\dot{x}
= \left( -\dfrac{\dot{A}_1}{A_1} + \dfrac{1}{A_1}- \dfrac{A_1}{B_1^2}  \right) x + y^2 -x^2,
\end{equation}
\begin{equation}\label{eq:y}
\dot{y}
= \left( -\dfrac{\dot{B}_1}{B_1} - \dfrac{2}{A_1} + 2x \right) y.
\end{equation}
\end{prop}
\begin{proof}
By the same argument as in \cite[Proposition 5]{Lotay17}\label{prop:5}, we may always write $\theta$ as in (\ref{eq:3.2}). Then, the constraints from Lemma \ref{lemma:3} hold, and the ODEs may be written in two different ways.
First, we observe that if we write
\begin{equation}\label{eq:3.2b}
\theta=x^+ \left( \sum_{i=1}^3 T_i \otimes \eta_i^+ \right) + x^- \left( \sum_{i=1}^3 T_i \otimes \eta_i^- \right),
\end{equation}
then we obtain the following ODEs:
\begin{equation}\label{eq:prevx+}
\dot{x}^+ = \dfrac{x^+}{A_1} \left( 1 -\dfrac{A_1^2}{B_1^2} -x^+ \right) + \dfrac{A_1}{B_1^2} (x^-)^2,
\end{equation}
\begin{equation}\label{eq:prevx-}
\dot{x}^- = \dfrac{2 x^-}{A_1} (x^+ -1).
\end{equation}
This way of writing the equations will be useful in the future.
Now from the relation
$$
x^+ = x A_1, \quad x^-= y B_1,
$$
the system (\ref{eq:prevx+}), (\ref{eq:prevx-}) becomes (\ref{eq:x}), (\ref{eq:y}).
\end{proof}

\begin{comment}
$$
\dfrac{2\dot{A}_1 -3}{A_1} =\dfrac{-4}{t} +20a t
$$
$$
\dfrac{\dot{A}}{A_1} =\dfrac{1}{t} +4a t
$$
\end{comment}

The next Lemma, whose prove uses Lemma \ref{lemma:10}, tells us when the $G_2$-instantons extend smoothly over the singular orbit $S^3=\text{SU}(2)^2 / \Delta \text{SU}(2)$.
The conditions for the smooth extension will depend on the choice of bundle.

\begin{lemma}\cite[Lemma 4]{Lotay17}\label{lemma:extensionxy}
The connection $\theta$ in equation (\ref{eq:3.2}) extends smoothly over the singular orbit $S^3$ if $x(t)$ is odd, $y(t)$ is even, and their Taylor expansions around $t=0$ are
\begin{enumerate}[label=(\roman*)]
    \item either $x(t)= x_1 t +x_3 t^3 +...$, $y(t)= y_2 t^2 +...$, in which case $\theta$ extends smoothly as a connection on $P_1$;
    \item or $x(t)=\frac{2}{t} + x_1 t +...$, $y(t)=y_0 + y_2 t^2 +...$, in which case $\theta$ extends smoothly as a connection on $P_{\text{id}}$.
\end{enumerate}
\end{lemma}

\subsection{Coclosed \texorpdfstring{$\text{SU}(2)^3$}{SU(2)3}-invariant \texorpdfstring{$G_2$}{G2}-instantons}\label{sec:SU(2)^3coclosed}

In this section we focus on the case where the $G_2$-structures are $\text{SU}(2)^3$-invariant and coclosed but not necessarily torsion-free.
Let $A_1:[0, \infty) \rightarrow \R$ be a smooth function with $A_1(t) >0$ for $t \in (0,\infty)$, such that
\begin{enumerate}[label=(\roman*)]
    \item $A_1$ is odd around $t=0$;
    \item $\dot{A}_1(0)=1/2$.
\end{enumerate}
We denote
$$
A_1(t)=\dfrac{t}{2}+a_{1,3} t^3 + O(t^5).
$$
Consider the coclosed $G_2$-structure found in \cite[Theorem 4.9]{Alo22} for $A_1=A_2=A_3$. By \cite[Section 3.4.2]{Alo22}, we have that the functions $B_1, B_2, B_3$ that define the $G_2$-structure are all equal, so the $G_2$-structure presents an extra $\text{SU}(2)$-symmetry.

\begin{remark}
There is a complicated but explicit expression of $B_1$ in terms of $A_1$ and $b_0$ (see \cite[Section 3.4.1]{Alo22}):
\begin{equation}
    B_1(t)=\sqrt{A_1^{-2}(t) \left( \dfrac{b_0^2}{16} e^{\int_{1/2}^t \frac{1}{A_1(\xi)}d\xi } + e^{\int_{1/2}^t \frac{1}{A_1(\xi)} d\xi} \int_0^t A_1^3(\eta) e^{- \int_{1/2}^\eta \frac{1}{A_1(\xi)} d\xi} d\eta \right) }.
\end{equation}
We also recall that
$$
B_1(t)=b_0 + b_2 t^2 + O(t^4).
$$
\end{remark}

We need to divide our discussion depending on which choice of bundle over a singular orbit we take: $P_1$ or $P_{\text{id}}$ (see Section \ref{sec:extension}).

\subsubsection{Extension on \texorpdfstring{$P_1$}{P1}}

For the manifold $\R^4 \times S^3$ with any of the coclosed $G_2$-structures from \cite{Alo22}, we obtain a 1-parameter family of $G_2$-instantons extending over the singular orbit $P_1$.

\begin{theorem}\label{thm:coclosedClarke}
Let $M=\R^4 \times S^3$, with a $\text{SU}(2)^3$-invariant coclosed $G_2$-structure given by $A_1$ and $b_0 >0$ as in \cite[Theorem 4.9]{Alo22}. 
There is an explicit 1-parameter family of $\text{SU}(2)^3$-invariant $G_2$-instantons with gauge group $\text{SU}(2)$ on the bundle $P_1$, given by
\begin{equation}\label{eq:Ax1}
\theta^{x_1}= \dfrac{x_1 A_1 e^{\int_{1/2}^t F(\xi) d\xi}}{1 + x_1 \int_0^t e^{\int_{1/2}^\eta F(\xi) d\xi} d\eta} \sum_{i=1}^3 T_i \otimes \eta_i^+,
\end{equation}
where
$$
F(t)=-\dfrac{\dot{A}_1}{A_1} + \dfrac{1}{A_1}- \dfrac{A_1}{B_1^2},
$$
and $x_1 \in [0, \infty)$.
Given such $x_1$ we denote the resulting instanton by $\theta^{x_1}$, and $\theta^0$ is the trivial flat connection.
Moreover, any $\text{SU}(2)^3$-invariant $G_2$-instantons with gauge group $\text{SU}(2)$ on the bundle $P_1$ are in this family.
\end{theorem}
\begin{proof}
From Lemma \ref{lemma:extensionxy} we see that for the connection $\theta$ to smoothly extend to the singular orbit on the bundle $P_1$, we need $u, v: [0,\infty) \rightarrow \R$ real analytic even functions such that
$$
\begin{array}{ll}
x(t)=x_1 t +t^3 u(t), \\
y(t)=t^2 v(t).
\end{array}
$$
Here we are using the notation of Proposition \ref{prop:5}.
Then the system (\ref{eq:x}), (\ref{eq:y}) gives
\begin{equation}\label{eq:uv}
\begin{array}{ll}
\dot{u}=\dfrac{-2u -x_1^2-x_1(8 a_{1,3}+ 1/2b_0^2)}{t} + f_1(t,u,v), \\
\dot{v}=\dfrac{-6v}{t} + f_2(t,u,v),
\end{array}
\end{equation}
where $f_1, f_2: [0,\infty) \times \R^2 \rightarrow \R$ are some real analytic functions.
Theorem \cite[Theorem 4.7]{FH17} guarantees the existence and uniqueness of solutions to this system in a neighbourhood of the initial value $t=0$ provided that
\begin{equation}\label{eq:u0v0}
\begin{array}{ll}
u(0)=\dfrac{-x_1^2}{2} - x_1 \left( 4 a_{1,3} + \dfrac{1}{4 b_0^2} \right), \\
v(0)=0.
\end{array}
\end{equation}
Suppose $y=0$. Then equation (\ref{eq:x}) becomes
\begin{equation}\label{eq:x0}
\dot{x}
= \left( -\dfrac{\dot{A}_1}{A_1} + \dfrac{1}{A_1}- \dfrac{A_1}{B_1^2}  \right) x  -x^2.
\end{equation}
Let
$$
F(t)=-\dfrac{\dot{A}_1}{A_1} + \dfrac{1}{A_1}- \dfrac{A_1}{B_1^2},
$$
which has the following Taylor expansion at 0:
\begin{equation}\label{eq:Fexpansion}
F(t)= \dfrac{1}{t} + \left( -8 a_{1,3}- \dfrac{1}{2b_0^2} \right) t +O(t^3).
\end{equation}
The equation
\begin{equation}\label{eq:Bernoulli}
\dot{x}=F(t) x -x^2
\end{equation}
is a Bernoulli differential equation, and hence can be solved by making the change of variables $z=x^{-1}$ and later using an integrating factor. By a straightforward computation we get 
that all the solutions to (\ref{eq:Bernoulli}) are $x \equiv 0$ or 
\begin{equation}\label{eq:xk}
x(t)= \dfrac{ e^{\int_{1/2}^t F(\xi) d\xi}}{k+ \int_0^t e^{\int_{1/2}^\xi F(\eta) d\eta} d\xi},
\end{equation}
for some choice of real constant $k$. Note that choosing $k$ is the same as choosing the limits of the integrals. 
Suppose that $k \neq 0$. Then we can write
\begin{equation}\label{eq:1-paramBS}
x(t)= \dfrac{x_1 e^{\int_{1/2}^t F(\xi) d\xi}}{1 + x_1 \int_0^t e^{\int_{1/2}^\eta F(\xi) d\xi} d\eta},
\end{equation}
for some real $x_1$.
This expression, together with $y=0$, agrees with the one corresponding to the unique solution of (\ref{eq:uv}) in a neighbourhood of $t=0$, as a straightforward computation shows that (\ref{eq:u0v0}) holds. 
If we take $x_1 \geq 0$, this explicit expression is well defined for $t >0$,
so the resulting instantons are defined globally on $\R^4 \times S^3$.
It remains to check that (\ref{eq:1-paramBS}) extends smoothly to the singular orbit at $t=0$ as a connection on $P_1$. 
Let $0< t \ll1$, and fix $a \ll 1$, $a>t$. We can write
$$
x(t) 
= \dfrac{x_1 \exp \left( \int_{1/2}^a F(\xi) d\xi \right) \exp \left(\int_a^{t} F(\xi) d\xi \right)}{ 1+ x_1\int_0^{t} \exp \left( \int_{1/2}^a F(\xi) d\xi \right) \exp \left( \int_a^\eta F(\xi) d\xi \right) d\eta}.
$$
We observe that $\int_{1/2}^a F(\xi) d\xi$ is constant, and denote it by $c$.
Then for $\xi \in [t, a]$, we can approximate $F(\xi)\cong \xi^{-1}$.
Then
\begin{equation}\label{eq:xexpansionP1}
\begin{array}{ll}
x(t)
& \cong \dfrac{x_1 e^c \exp \left(\int_a^{t} \xi^{-1} d\xi \right) }{1+ x_1 e^c \int_0^{t} \exp \left( \int_a^\eta \xi^{-1} d\xi \right)  d\eta}
= \dfrac{x_1 e^c a^{-1} t }{1+ x_1 e^c a^{-1} t^2 /2}.
\end{array}
\end{equation}
Hence, $\lim_{t \rightarrow 0} x t^{-1} \in \R$. Finally, the fact that $x$ is odd follows from $F(t)= 1/t + O(t)$ and $F$ odd.
\end{proof}

\begin{remark}
We recover Clarke's instantons \cite{Clarke14} 
\begin{equation}\label{eq:Clarkeins}
\theta^{x_1}= \dfrac{2 x_1 A_1(t)^2}{1+x_1 (B_1(t)^2-1/3)} \sum_{i=1}^3 T_i \otimes \eta_i^+,
\end{equation}
when $A_1=\dfrac{r}{3} \sqrt{1-r^{-3}}$, $b_0 =\dfrac{1}{\sqrt{3}}$ and $r \in [1, + \infty)$ is a coordinate defined implicitly by
$
t(r)= \int_1^r \dfrac{ds}{\sqrt{1-s^{-3}}}.
$
These were the first examples of $G_2$-instantons on $\R^4 \times S^3$ with the Bryant Salamon $G_2$-holonomy metric.
\end{remark}

One may wonder what is the asymptotic behaviour of $\theta^{x_1}$ when $t \rightarrow \infty$. However, this would depend on the asymptotic behaviour of the data $A_1$, and we have not introduced any a priori restrictions on it.

We can compute the curvature of these instantons using (\ref{eq:FA}) and \cite[Lemma 2]{Lotay17}:
$$
F_{\theta^{x_1}}= T_i \otimes \left( \dfrac{d}{dt} (A_1 x) dt \wedge \eta_i^+ +A_1 x(A_1 x-1) \epsilon_{ikj} \eta_j^+ \wedge \eta_k^+ -A_1 x \epsilon_{ijk} \eta_j^- \wedge \eta_k^- \right).
$$
We deduce that %from (\ref{eq:x(1-e)}) that 
$$
\lim_{t \rightarrow 0} F_{\theta^{x_1}} = \lim_{t \rightarrow 1} F_{\theta^{x_1}} = - \epsilon_{ijk} T_i \otimes  \eta_j^- \wedge \eta_k^-.
$$
In particular, the curvature is bounded at the singular orbit.

\subsubsection{Extension on \texorpdfstring{$P_\text{id}$}{Pid}}

We now study the existence of smooth $G_2$-instantons on the bundle $P_\text{id}$ at the singular orbit, and its extension away from the singular orbit.

\begin{prop}\label{prop:Pid}
Let $S^3$ be a singular orbit in $\R^4 \times S^3$, with coclosed $G_2$-structure given by $A_1= A_2 =A_3$ as in \cite[Theorem 4.9]{Alo22}. 
There is exactly a 1-parameter family of $\text{SU}(2)^3$-invariant $G_2$-instantons, with gauge group $\text{SU}(2)$, defined in a neighbourhood of $S^3$ and smoothly extending over $S^3$ on $P_{\text{id}}$.
\end{prop}
\begin{proof}
We see from Lemma \ref{lemma:extensionxy} that for the connection $\theta$ to smoothly extend to a singular orbit at $t=0$, we need $u, v: [0,\infty) \rightarrow \R$ real analytic
functions, such that 
\begin{equation}\label{eq:x,y(u,v)}
\begin{array}{ll}
x(t)=\dfrac{2}{t} +t u(t), \\
y(t)=y_0 +t^2 v(t).
\end{array}
\end{equation}
Then the system (\ref{eq:x}), (\ref{eq:y}) gives
\begin{equation}\label{eq:uvid}
\begin{array}{ll}
\dot{u}=\dfrac{-4u + y_0^2 -16a_{1,3} -1/b_0^2}{t} + f_1(t,u,v), \\
\dot{v}=\dfrac{-2v +2 y_0 u +y_0 (8 a_{1,3}-2b_2/b_0)}{t} + f_2(t,u,v),
\end{array}
\end{equation}
where $f_1, f_2: [0,\infty) \times \R^2 \rightarrow \R$ are some real analytic functions.
Theorem \cite[Theorem 4.7]{FH17} guarantees the existence and uniqueness of solutions to this system in a neighbourhood of the initial value $t=0$ provided that
\begin{equation}\label{eq:u0v0id}
\begin{array}{ll}
u(0)=\dfrac{y_0^2}{4} - 4 a_{1,3} - \dfrac{1}{4b_0^2}, \\
v(0)=\dfrac{y_0^3}{4} - y_0 \left( \dfrac{1}{4b_0^2} -\dfrac{b_2}{b_0} \right).
\end{array}
\end{equation}
Therefore $(x(t), y(t))$ given by (\ref{eq:x,y(u,v)}) provide a solution of (\ref{eq:x}) and (\ref{eq:y}). Both $F(t)$ and $G(t)= -\dot{B}_1(t) /B_1(t) -2/A_1(t)$ are odd. We deduce that $(\tilde{x}(t), \tilde{y}(t)) \coloneqq (-x(-t), y(-t))$ is also a solution of (\ref{eq:x}), (\ref{eq:y}). Note that although $F$ and $G$ are only defined for $t>0$, we can extend them on $t<0$ as odd functions.
We can write 
$$
\begin{array}{ll}
\tilde{x}(t)=\dfrac{2}{t} +t \tilde{u}(t), \\
\tilde{y}(t)=y_0 +t^2 \tilde{v}(t),
\end{array}
$$
for real analytic $\tilde{u}, \tilde{v}$ with $\tilde{u}(0)=u(0)$ and $\tilde{v}(0)=v(0)$, and by uniqueness $\tilde{u}=u$, $\tilde{v}=v$. Therefore, we can smoothly extend $(x,y)$ to $t<0$ by $(x(-t), y(-t))=(-x(t), y(t))$ and they still solve (\ref{eq:x}) and (\ref{eq:y}), which gives that $x$ is odd, $y$ is even as desired.
Hence, for each $y_0$ we have a smooth $G_2$-instanton in a neighbourhood of the singular orbit.
\end{proof}

When these instantons can be extended away from the singular orbit, they will be denoted by $\theta_{y_0}$, where $y_0$ is the initial parameter $y(0)$. 

If $y_0=0$, we can get an explicit expression of the $G_2$-instanton. 

\begin{theorem}\label{thm:Alim}
There is an explicit $\text{SU}(2)^3$-invariant $G_2$-instanton with gauge group $\text{SU}(2)$ on $M=\R^4 \times S^3$ with coclosed $G_2$-structure given by $A_1= A_2 =A_3$ as in \cite[Theorem 4.9]{Alo22} corresponding to $y_0=0$ in Proposition \ref{prop:Pid}.
It is given by
\begin{equation}\label{eq:inscompact}
\theta_0 = \dfrac{A_1(t) e^{\int_{1/2}^t F(\xi) d\xi}}{\int_0^t e^{\int_{1/2}^\eta F(\xi) d\xi} d\eta} \sum_{i=1}^3 T_i \otimes \eta_i^+.
\end{equation}
where
$$
F(t)=-\dfrac{\dot{A}_1}{A_1} + \dfrac{1}{A_1}- \dfrac{A_1}{B_1^2}.
$$
It smoothly extends on the bundle $P_{id}$ over $\R^4 \times S^3$. 
\end{theorem}
\begin{proof}
Here we are also using the notation of Proposition \ref{prop:5}.
%We first see that $y=0$, and then solve the ODE for $x$ (\ref{eq:x0}) with initial condition $x(0) A_1 (0)=1$. We get that the solution is the one from (\ref{eq:xk}) with $k=0$:
We first see that $y=0$ gives a solution, which corresponds to the value of the parameter $y_0=0$, by taking the solution $x_0$ from (\ref{eq:xk}) with $k=0$:
$$
x_0(t)= \dfrac{e^{\int_{1/2}^t F(\xi) d\xi}}{\int_0^t e^{\int_{1/2}^\eta F(\xi) d\xi} d\eta}.
$$
We will prove that it satisfies that $x_0(0) A_1 (0)=1$. 
Let $0<t \ll 1$, and fix $a \ll 1$, $a>t$. We can write
$$
x_0(t) A_1(t)
= \dfrac{A_1(t) \exp \left( \int_{1/2}^a F(\xi) d\xi \right) \exp \left(\int_a^{t} F(\xi) d\xi \right)}{ \int_0^{t} \exp \left( \int_{1/2}^a F(\xi) d\xi \right) \exp \left( \int_a^\eta F(\xi) d\xi \right) d\eta}
= \dfrac{A_1(t) \exp \left(\int_a^{t} F(\xi) d\xi \right)}{ \int_0^{t} \exp \left( \int_a^\eta F(\xi) d\xi \right) d\eta},
$$
Then for $\xi \in [t, a]$, we can approximate $F(\xi)$ by $\xi^{-1}$, so
$$
\begin{array}{ll}
x_0(t) A_1(t)
\cong \dfrac{\dfrac{t}{2} \exp \left(\int_a^{t} \xi^{-1} d\xi \right)}{ \int_0^{t} \exp \left( \int_a^\eta \xi^{-1} d\xi \right) d\eta} 
%= \dfrac{\dfrac{t}{2}t a^{-1}}{ \int_0^{t} \eta a^{-1} d\eta}
= \dfrac{t^2 a^{-1}/2}{a^{-1} \int_0^{t} \eta  d\eta}
= \dfrac{t^2 a^{-1}/2}{t^2 a^{-1}/2}
=1.
\end{array}
$$
Another computation shows that (\ref{eq:u0v0id}) holds.
Furthermore, $x_0(t)$ being odd follows from $F(t)=1/t +O(t)$ and $F$ odd. 
Hence, as the explicit expression is well defined for all $t \in [0, \infty)$, and the instanton smoothly extends to the singular orbit at $t=0$, the instanton (\ref{eq:inscompact}) extends to the whole manifold.
\end{proof}

\begin{remark}
When $M=\R^4 \times S^3$, $A_1=\dfrac{r}{3} \sqrt{1-r^{-3}}$, $b_0 =\dfrac{1}{\sqrt{3}}$ and $r \in [1, + \infty)$ is a coordinate defined implicitly by
$
t(r)= \int_1^r \dfrac{ds}{\sqrt{1-s^{-3}}},
$
we recover the $G_2$-instanton $A^{\text{lim}}$ from \cite[Theorem 5]{Lotay17}:
\begin{equation}\label{eq:AlimBS}
A^\text{lim} = \dfrac{A_1(t)^2}{1/2 (B_1(t)^2-1/3)} \sum_{i=1}^3 T_i \otimes \eta_i^+.
\end{equation}
\end{remark}

We now move on to study the extension of instantons on $P_\text{id}$ for other values of the parameter $y_0$. Unlike when $y_0=0$, we will not have explicit expressions.
Recall that if we perform the change $x^+ = x A_1$, $x^-= y B_1$, then (\ref{eq:x}) and (\ref{eq:y}) become
\begin{equation}\label{eq:x+}
\dot{x}^+ = \dfrac{x^+}{A_1} \left( 1 -\dfrac{A_1^2}{B_1^2} -x^+ \right) + \dfrac{A_1}{B_1^2} (x^-)^2, 
\end{equation}
\begin{equation}\label{eq:x-}
\dot{x}^- = \dfrac{2 x^-}{A_1} (x^+ -1).
\end{equation}
We observe that the critical points of this system of ordinary differential equations are $(0,0)$, $(1,1)$ and $(1, -1)$. The point $(0,0)$ corresponds to the flat connection $\theta=0$.
The other points, $(1, 1)$ and $(1,-1)$, correspond to connections which are smooth on $P_\text{id}$, and that in the notation of Proposition \ref{prop:Pid} correspond to values of the parameter $y_0$ of $1/b_0$ and $-1/b_0$, respectively. The corresponding connections, which we denote by $\theta_{1/b_0}$ and $\theta_{-1/b_0}$, are:
\begin{equation}\label{eq:flatinst}
\theta_{1/b_0} = \sum_{i=1}^3 T^i \otimes \eta_i^+ + \sum_{i=1}^3 T^i \otimes \eta_i^-, \quad  
\theta_{-1/b_0} = \sum_{i=1}^3 T^i \otimes \eta_i^+ - \sum_{i=1}^3 T^i \otimes \eta_i^-.
\end{equation}
A quick computation shows that both of these connections are flat. 
Note that all flat connections can by identified by (non-invariant) gauge transformations.
We will show that other values of the parameter $y_0$ also give smooth instantons on our manifolds of interest, although no longer explicit.

We say that a subset $R \subset \R^n$ is \textit{forward-invariant} for an ODE system $\dot{x} = F (x, t)$ if for a solution $x(t)$ and a non-singular time $t_0$ (i.e.\ the ODE is regular at $t_0$) such that $x(t_0) \in R$, then $x(t) \in R$ for any $t>t_0$ such that $x(t)$ exists. 
We will use forward-invariance of sets to show that the non-autonomous system of ODEs (\ref{eq:x+}), (\ref{eq:x-}), with initial point $(x^+(0), x^-(0))=(1, y_0 b_0)$, has a solution away from the singular orbit.

\begin{prop}\label{prop:1-paramPid}
For values of the parameter $y_0$ in $[-1/b_0, 1/b_0]$, the $G_2$-instantons from Proposition \ref{prop:Pid} extend to the manifold $\R^4 \times S^3$. We denote these instantons by $\theta_{y_0}$.
\end{prop}
\begin{proof}
We already considered the case $y_0=0$ on Theorem \ref{thm:Alim}, and $y_0=\pm 1/b_0$ corresponds to the flat instantons (\ref{eq:flatinst}).
Separating variables in (\ref{eq:x-}), we can solve it to get
\begin{equation}\label{eq:x-solution}
x^-(t) = x^-(0) \exp \left( \int_0^t \frac{2}{A_1(\xi)}(x^+(\xi) -1)d\xi \right).
\end{equation}
Therefore, away from the singular orbits, $y$ is either always positive, or always negative, or identically 0.
We observe that $(x,y) \mapsto (x, -y)$ is a symmetry of equations (\ref{eq:x}) and (\ref{eq:y}). Hence, we may assume $y>0$ or equivalently $x^- >0$ and study the case $y_0 \in (0, 1/b_0)$. The remaining case $y_0 \in (-1/b_0, 0)$ will follow by changing the signs of $y$.
We will show that the following set is forward-invariant:
$$
R \coloneqq \{ (x^+, x^-) \in \R^2 \ \vert \ 0 < x^+ < 1, 0 < x^- < 1 \}.
$$
There exists $t_0 >0$ such that the solution $(x^+, x^-)$ of (\ref{eq:x+}) and (\ref{eq:x-}) is defined on $[0, t_0]$ and $(x^+(t_0), x^-(t_0)) \in R$.
%We used that we can always assume $x^- > 0$ by Lemma \ref{lemma:symmetry}. 
First, we see that $\dot{x}^- < 0$ when $x^+ <1$, and $\{ (x_+, 0) \vert 0 \leq x^+ \leq 1 \}$ is an invariant line for the flow. We also observe that for $t>0$
$$
\dot{x}^+ \vert_{x^+ =1} =\dfrac{A_1}{B_1^2}((x^-)^2 -1) <0,
$$
and
$$
\dot{x}^+ \vert_{x^+ =0} 
= \dfrac{A_1}{B_1^2} (x^-)^2 >0,
$$
when $0 < x^- < 1$, as $A_1 >0$, so $\{ (1, x_-) \vert 0 \leq x^- \leq 1 \}$ and $\{ (0, x_-) \vert 0 \leq x^- \leq 1 \}$ are also invariant lines for the flow.
Hence, the set $R$ is forward-invariant and $(x^+, x^-)$ cannot blow-up, so the instanton given by  $(x^+, x^-)$ is defined for every $t \in [0, \infty)$.
\end{proof}

\begin{remark}
Once again, taking $A_1=\dfrac{r}{3} \sqrt{1-r^{-3}}$, $b_0 =\dfrac{1}{\sqrt{3}}$ and $r \in [1, + \infty)$ a coordinate defined implicitly by
$
t(r)= \int_1^r \dfrac{ds}{\sqrt{1-s^{-3}}},
$
we can recover the corresponding 1-parameter family $G_2$-instantons on $\R^4 \times S^3$ with the Bryant--Salamon metric; for this case its existence was shown in \cite[Theorem 3.7, Lemma 3.9]{SteinTurner}, where it is denoted as $T'_{\gamma'}$, $\gamma' \in [-1, 1]$.
\end{remark}

\begin{remark}
    We can go from $\theta_{y_0}$ to $\theta_{-y_0}$ by exchanging the factors of $\text{SU}(2)^2$.
\end{remark}

Putting everything together, we get the following theorems.

\begin{theorem}
Let $M=\R^4 \times S^3$, with a $\text{SU}(2)^3$-invariant coclosed $G_2$-structure given by $A_1$ and $b_0 >0$ as in \cite[Theorem 4.9]{Alo22}. 
There exists two 1-parameter families of smooth $\text{SU}(2)^3$-invariant $G_2$-instantons with gauge group $\text{SU}(2)$: $\theta^{x_1}$, $x_1 \in [0, \infty)$ extending to $P_1$; and $\theta_{y_0}$, $y_0 \in [-1/b_0, 1/b_0]$ extending to $P_\text{id}$.
\end{theorem}

One may hope to make this result into a classification result, by studying the remaining possible situation, corresponding the values of the parameter $y_0$ with $\vert y_0 \vert > 1/b_0 $. In \cite{SteinTurner} it is shown that for $(\R^4 \times S^3, g_{BS})$, solutions corresponding to these initial parameter do not produce uniformly bounded instantons.

\subsection{Behaviour of solutions}\label{sec:behaviour}

We observe the expected bubbling behaviour of the sequence of instantons $\theta^{x_1}$: they ``bubble off'' an anti-self-dual (ASD) connection along the normal bundle to the singular orbit $S^3 = \{ 0 \} \times S^3 \subset \R^4 \times S^3$, which is an associative submanifold.

\begin{theorem}
Let $\theta^{x_1}$, $x_1 \geq 0$ be the sequence of instantons from Theorem \ref{thm:coclosedClarke}, and $\theta_0$ be the instanton from Theorem \ref{thm:Alim}. Then
\begin{enumerate}[label=(\roman*)]
    \item For any $\lambda >0$ there is a sequence $\delta =\delta(x_1, \lambda)>0$ converging to 0 when $x_1 \rightarrow \infty$ such that: for all $p \in S^3$, and if we define
    \begin{equation*}
    \begin{split}
        s_\delta^p : B_1 \subset \R^4 &\rightarrow B_\delta \times \{ p \} \subset \R^4 \times S^3;\\
        x & \mapsto (\delta x, p ),
    \end{split}
\end{equation*}
    then $(s_\delta^p)^* \theta^{x_1} $ converges uniformly with all derivatives to the basic ASD instanton $\theta_\lambda^\text{ASD}$ with scale $\lambda$ on $B_1 \subset \R^4$:
    $$
    \theta_\lambda^{\text{ASD}}= \dfrac{\lambda t^2}{1+ \lambda t^2} \sum_{i=1}^3 T_i \otimes \eta_i^+.
    $$
    \item 
    Suppose that $\theta_0$ is bounded.
    The connections $\theta^{x_1}$ converge to $\theta_0$ given in Theorem \ref{thm:Alim} on every compact subset of $(\R^4 \setminus \{ 0 \}) \times S^3$ when $x_1 \rightarrow \infty$.
\end{enumerate}
\end{theorem}
\begin{proof}
(i) %Recall that near $t=0$, $A_1(t)=t/2 +O(t^3)$ and $F(t)=1/t +O(t^3)$, so we also have $e^{\int F(t) dt}=t +O(t^3)$. Then we can compute
It follows from (\ref{eq:xexpansionP1}) that there is some constant $c_{x_1}$, increasing with $x_1$, such that  
$$
\begin{array}{ll}
(s_\delta^p)^* \theta^{x_1} 
&= A_1(\delta t) x(\delta t) \sum_{i=1}^3 T_i \otimes \eta_i^+ \\
&=\dfrac{c_{x_1} \delta^2 t^2 /2 +O(x_1 \delta^4 t^4)}{1+ c_{x_1} \delta^2 t^2 /2 +O(x_1 \delta^4 t^4)} \sum_{i=1}^3 T_i \otimes \eta_i^+ \\
\end{array}
$$
By taking $\delta =\delta (x_1, \lambda) = \sqrt{2\lambda /c_{x_1}} >0$, we have that $\delta (x_1, \lambda) \rightarrow 0$ when $x_1 \rightarrow \infty$ and for every $k \in \mathbb{N} \cup \{ 0 \}$, there is a $c_k >0$, not depending on $\lambda, x_1$, such that 
$$
\vert \vert (s_\delta^p)^* \theta^{x_1} - \theta_\lambda^\text{ASD} \vert \vert_{C^k (B_1)} \leq c_k \dfrac{\lambda^2}{c_{x_1}}.
$$
The uniform convergence with all derivatives follows. \\
(ii) We recall that $\theta^{x_1}$ is given by (\ref{eq:Ax1}) and $\theta_0$ by (\ref{eq:inscompact}), and observe that
$\theta^{x_1}$ converges pointwise to $\theta_0$.
We have
$$
\begin{array}{ll}
\vert \theta^{x_1}-\theta_0 \vert
&= A_1 e^{\int_{1/2}^t F(\xi) d\xi} \bigg\rvert \dfrac{x_1}{1 + x_1 \int_0^t e^{\int_{1/2}^\eta F(\xi) d\xi} d\eta} - \dfrac{1}{\int_0^t e^{\int_{1/2}^\eta F(\xi) d\xi} d\eta} \bigg\rvert \bigg\rvert \sum_{i=1}^3 T_i \otimes \eta_i^+ \bigg\rvert \\
&= A_1 e^{\int_{1/2}^t F(\xi) d\xi} \bigg\rvert \dfrac{1}{(1 + x_1 \int_0^t e^{\int_{1/2}^\eta F(\xi) d\xi} d\eta) \int_0^t e^{\int_{1/2}^\eta F(\xi) d\xi} d \eta} \bigg\rvert \bigg\rvert \sum_{i=1}^3 T_i \otimes \eta_i^+ \bigg\rvert \\
&= \vert \theta_0 \vert \bigg\rvert \dfrac{1}{1 + x_1 \int_0^t e^{\int_{1/2}^\eta F(\xi) d\xi} d\eta} \bigg\rvert \\
\end{array}
$$
Let $c_1>0$ be a bound for $\theta_0$. A quick computation (similar to the one at the end of the proof of Theorem \ref{thm:coclosedClarke}) shows that on every compact subset of $(\R^4 \setminus \{ 0 \}) \times S^3$, there exists a constant $c_2 >0$ such that $c_2 \leq \int_0^t e^{\int_{1/2}^\eta F(\xi) d\xi} d\eta$.
Then 
$$
\begin{array}{ll}
\vert \theta^{x_1}-\theta_0 \vert \leq \dfrac{c_1}{1+x_1 c_2}.
\end{array}
$$
Therefore $\theta^{x_1} - \theta_0$ converges uniformly to zero when $x_1 \rightarrow \infty$.
Similarly the derivatives of $\theta^{x_1}- \theta_0$ converge uniformly to zero when $x_1 \rightarrow \infty$.
\end{proof}

\begin{remark}
We are interested in connections that have, at the very least, bounded curvature. 
Hence the condition that $\theta_0$ is bounded on (ii) of the previous theorem is reasonable, as we are already looking for $A_1$ functions that make our instantons bounded.
\end{remark}

\begin{ex}
For $A_1(t)=t/2$, we get
$$
\theta^{x_1}= \dfrac{x_1 t/2}{t^2/4 +c}  \dfrac{1}{1+2 x_1 \log(t^2/4 +c)} \sum_{i=1}^3 T^i \otimes \eta_i^+, \quad 
\theta_0= \dfrac{t^2/4}{t^2/4 +c}  \dfrac{1}{\log(t^2/4+c)} \sum_{i=1}^3 T^i \otimes \eta_i^+.
$$
Then $\theta^{x_1} \rightarrow 0$, $\theta_0 \rightarrow 0$ when $t \rightarrow \infty$. 
\end{ex}

Figure 1 represents the instantons found on $\R^4 \times S^3$ for the $G_2$-structures of \cite{Alo22} and their relations. We note that this general representation of instantons also generalises the one of $(\R^4 \times S^3, g_{BS})$, which is depicted in \cite{SteinTurner}.

\begin{figure}\label{figure:instantons}
\begin{center}
\begin{tikzcd}
                                                           &                                                                                                          & \theta_{b_0^{-1}}             \\
\theta \equiv 0 \arrow[r, "\theta^{x_1}"] & \theta_0 \arrow[ru, "\theta_{y_0}", no head, bend left] \arrow[rd, "\theta_{y_0}"', no head, bend right] &                               \\
                                                           &                                                                                                          & \theta_{-b_0^{-1}} \arrow[uu, leftrightarrow, shift left=10]
\end{tikzcd}
\caption{Representation of families of $G_2$-instantons and their relations.}
\label{spheres}
\end{center}
\end{figure}
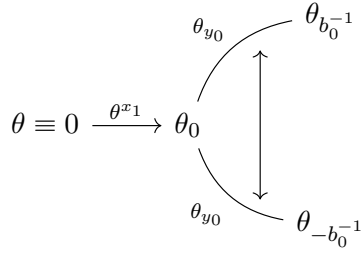

\section{\texorpdfstring{$\text{SU}(2)^2$}{SU(2)2}-invariant \texorpdfstring{$G_2$}{G2}-instantons}\label{sec:SU(2)^2ins}

In this section we consider the most general case, when the coclosed $G_2$-structures are $\text{SU}(2)^2$-invariant. We provide an existence and classification results for $G_2$-instantons on neighbourhoods of the singular orbit. 
Let $A_1, A_2, A_3:[0, \infty) \rightarrow \R$ be smooth functions with $A_i(t) >0$ for $t \in (0,\infty)$, such that
\begin{enumerate}[label=(\roman*)]
    \item $A_i$'s are odd;
    \item $\dot{A}_i(0)=1/2$.
\end{enumerate}
We denote
\begin{equation}\label{eq:tayexpO(5)}
%A_i(t) =\dfrac{t}{2} + a_{i,3} t^3 + O(t^5).
A_i(t)=\dfrac{t}{2} +a_{i,3} t^3 +a_{i,5} t^5 +O(t^7).
\end{equation}
Consider the coclosed $G_2$-structure found in \cite[Theorem 4.9]{Alo22} for $A_1, A_2, A_3$.

\subsection{\texorpdfstring{$\text{SU}(2)^2$}{SU(2)2}-invariant ODEs and boundary conditions}\label{sec:SU(2)^2ODE}

We define
$$
c_i^+= \dfrac{a_i^+}{A_i}, \quad 
c_i^-= \dfrac{a_i^-}{B_i}.
$$
Then the general $\text{SU}(2)^2$-invariant $G_2$-instanton equations (\ref{eq:instantons}) from Lemma \ref{lemma:3} are:
\begin{equation}\label{eq:instantons2}
\begin{array}{ll}
\dot{c}_i^+ + \left( \dfrac{\dot{A}_i}{A_i} + \dfrac{A_i}{B_j B_k} - \dfrac{A_i}{A_j A_k}\right) c_i^+ = \dfrac{1}{2}[c_j^-, c_k^-] -\dfrac{1}{2}[c_j^+, c_k^+], \\
\dot{c}_i^- +\left( \dfrac{\dot{B}_i}{B_i}+ \dfrac{B_i}{B_j A_k} + \dfrac{B_i}{A_j B_k} \right) c_i^- = \dfrac{1}{2}[c_j^-, c_k^+] +\dfrac{1}{2}[c_j^+, c_k^-],
\end{array}
\end{equation}
$\{ i,j,k \}$ is a cyclic permutation of $\{ 1,2,3 \}$,
together with the constraint 
\begin{equation}
\sum_{i=1}^3 [c_i^+, c_i^-]=0.
\end{equation}

\begin{prop}\label{prop:Anotation}
Let $\theta$ be an $\text{SU}(2)^2$-invariant $G_2$-instanton on $\R^4 \times S^3$ with gauge group $\text{SU}(2)$. There is a standard basis $\{ T_i \}$ of $\mathfrak{su}(2)$ such that (up to an invariant gauge transformation) we can write
\begin{equation}\label{eq:Prop5.1}
\theta= \sum_{i=1}^3 A_i f_i^+ T_i \otimes \eta_i^+ + \sum_{i=1}^3 B_i f_i^- T_i \otimes \eta_i^- .
\end{equation}
with $f_i^\pm: [0, \infty) \rightarrow \R$ satisfying 
\begin{equation}\label{eq:fgh+-}
\begin{split}
\dot{f}_1^+ + \left( \dfrac{\dot{A}_1}{A_1} + \dfrac{A_1}{B_2 B_3} - \dfrac{A_1}{A_2 A_3} \right) f_1^+
&= f_2^- f_3^- - f_2^+ f_3^+, \\
\dot{f}_2^+ + \left( \dfrac{\dot{A}_2}{A_2} + \dfrac{A_2}{B_1 B_3} - \dfrac{A_2}{A_1 A_3} \right) f_2^+
&= f_1^- f_3^- - f_1^+ f_3^+, \\
\dot{f}_3^+ + \left( \dfrac{\dot{A}_3}{A_3} + \dfrac{A_3}{B_1 B_2} - \dfrac{A_3}{A_1 A_2} \right) f_3^+
&= f_1^- f_2^- - f_1^+ f_2^+, \\
\dot{f}_1^- + \left( \dfrac{\dot{B}_1}{B_1}+ \dfrac{B_1}{B_2 A_3} + \dfrac{B_1}{A_2 B_3} \right) f_1^-
&= f_2^- f_3^+ + f_2^+ f_3^-, \\
\dot{f}_2^- + \left( \dfrac{\dot{B}_2}{B_2}+ \dfrac{B_2}{B_1 A_3} + \dfrac{B_2}{A_1 B_3} \right) f_2^-
&= f_3^- f_1^+ + f_3^+ f_1^-, \\
\dot{f}_3^- + \left( \dfrac{\dot{B}_3}{B_3}+ \dfrac{B_3}{B_1 A_2} + \dfrac{B_3}{A_1 B_2} \right) f_3^-
&= f_2^- f_1^+ + f_2^+ f_1^-. \\
\end{split}
\end{equation}
\end{prop}
\begin{proof}
We must consider $\text{SU}(2)^2$-homogeneous $\text{SU}(2)$-bundles over a slice $S^3 \times S^3 \cong \text{SU}(2)^2 / \{ 1 \}$. Such bundles are parameterized by the trivial isotropy homomorphism $1: \{ 1 \} \rightarrow \text{SU}(2)$.
By Wang's Theorem \cite[Theorem 1]{Wang58}, invariant connections on the bundle $\text{SU}(2)^2 \times_{(\{ 1 \}, 1)} \text{SU}(2)$ can be written as a left-invariant extension 
$\Lambda: (\mathfrak{m}, \text{Ad}) \rightarrow (\mathfrak{su}(2), \text{Ad} \circ 1)$. 
Here $\mathfrak{m}$ splits into irreducibles as 
$$
\mathfrak{m}= \underbrace{\R \oplus ... \oplus \R}_{\text{6 times}}
$$
Therefore, we can apply a gauge transformation so that 
$$
a=
\sum_{i=1}^3 A_i f_i^+ T_i \otimes \eta_i^+  
+ \sum_{i=1}^3 B_i f_i^- T_i \otimes \eta_i^-,
$$
where $f_i^\pm$, $i=1,2,3$, are constants.
We then extend this connection to $\R^4 \times S^3$ and get
$$
a(t)
=\gamma \left( \sum_{i=1}^3 A_i f_i^+ T_i \otimes \eta_i^+  
+ \sum_{i=1}^3 B_i f_i^- T_i \otimes \eta_i^- \right) \gamma^{-1},
$$
for $\gamma: [0,\infty) \rightarrow \text{SU}(2)$ and $f_i^\pm: [0,\infty) \rightarrow \R$.
The constraint (\ref{eq:instconstraint}) is then satisfied. 
We also have that 
$$
\begin{array}{ll}
\dot{a}_i^+=(A_i f_i^+)^{\cdot} \gamma T_i \gamma^{-1} +A_i f_i^+ \gamma [\gamma^{-1} \dot{\gamma}, T_i] \gamma^{-1}, \\
\dot{a}_i^-=(B_i f_i^-)^{\cdot} \gamma T_i \gamma^{-1} +B_i f_i^- \gamma [\gamma^{-1} \dot{\gamma}, T_i] \gamma^{-1},
\end{array}
$$
so equation (\ref{eq:instantons}) implies that
%and the symmetry of (\ref{eq:instantons}) means that 
$$
A_i f_i^+ [\gamma^{-1} \dot{\gamma}, T_i]=0, \quad B_i f_i^- [\gamma^{-1} \dot{\gamma}, T_i]=0,
$$
$i=1,2,3$.
Finally, if $\theta \neq 0$ then $\dot{\gamma}=0$, so we may always find a gauge transformation such that $\theta$ is written as in (\ref{eq:Prop5.1}). Substituting this expression into (\ref{eq:instantons}) finishes the proof.
\end{proof}

The next Lemma deals with the conditions for the extension to the singular orbit, which is $\text{SU}(2)^2 / \Delta \text{SU}(2) \cong S^3$.

\begin{lemma}\label{lemma:extensionSU(2)^2}
The connection $\theta$ extends smoothly over the singular orbit $S^3$ if and only if $f_i^+$ are odd, $f_i^-$ are even, and their Taylor expansions around 0 are
\begin{itemize}
    \item either 
$$
\begin{array}{ll}
f_i^+=f_{i,1}^+ t +O(t^3), \quad f_i^-=f_{i,2}^- t^2 +O(t^4),
%f_1^+=f_{1,1}^+ t +O(t^3), \quad f_1^-=f_{1,2}^- t^2 +O(t^4),\\
%f_2^+=f_{2,1}^+ t +O(t^3), \quad f_2^-=f_{2,2}^- t^2 +O(t^4),\\
%f_3^+=f_{3,1}^+ t +O(t^3), \quad f_3^-=f_{3,2}^- t^2 +O(t^4),\\
\end{array}
$$
for $i=1,2,3$, in which case $\theta$ extends smoothly as a connection on $P_1$;
    \item or
$$
\begin{array}{ll}
f_i^+=\dfrac{2}{t}+ (b_2^+ - 4a_{i,3})t +O(t^3), \quad f_i^-=b_0^- +b_2^- t^2 +O(t^4),
%f_1^+=\dfrac{2}{t}+ (b_2^+ - 4a_{1,3})t +O(t^3), \quad f_1^-=b_0^- +b_2^- t^2 +O(t^4),\\
%f_2^+=\dfrac{2}{t}+ (b_2^+ - 4a_{2,3})t +O(t^3), \quad f_2^-=b_0^- +b_2^- t^2 +O(t^4),\\
%f_3^+=\dfrac{2}{t}+ (b_2^+ - 4a_{3,3})t +O(t^3), \quad f_3^-=b_0^- +b_2^- t^2 +O(t^4),\\
\end{array}
$$
for $i=1,2,3$, in which case $\theta$ extends smoothly as a connection on $P_{\text{id}}$.
\end{itemize}
\end{lemma}
\begin{proof}
For $P_1$, we apply \cite[Lemma 10]{Lotay17} to $\theta$ and get the first set of expressions above.
For $P_{\text{id}}$, we apply \cite[Lemma 10]{Lotay17} to
$$
\theta- \theta^{\text{can}}
= \sum_{i=1}^3 (A_i f_i^+ -1)T_i \otimes \eta_i^+ 
+ \sum_{i=1}^3 B_i f_i^- T_i \otimes \eta_i^-.
$$
We obtain the second set of expressions above.
\end{proof}

\subsection{Extension on \texorpdfstring{$P_1$}{P1}}\label{sec:P1}

We start by studying the existence of $G_2$-instantons around a singular orbit extending as a connection on $P_1$. We find a 3-parameter family of $G_2$-instantons defined in a neighbourhood of the singular orbit. 

\begin{prop}
Let $M= \R^4 \times S^3$, with coclosed $G_2$-structure given by $A_1, A_2, A_3$ as in \cite[Theorem 4.9]{Alo22}, and let $S^3 \subset M$ be the singular orbit.
There is a 3-parameter family of $\text{SU}(2)^2$-invariant $G_2$-instantons with gauge group $\text{SU}(2)$ in a neighbourhood of $S^3$ which smoothly extends in $P_1$. 
In the notation of Proposition \ref{prop:Anotation}, these instantons have $f_i^- =0$, $i=1,2,3$ and $f_1^+, f_2^+, f_3^+: (0,\infty) \rightarrow \R$ solve the ODEs
\begin{equation}\label{eq:fgh}
\begin{split}
\dot{f}_1^+ + \left( \dfrac{\dot{A}_1}{A_1} + \dfrac{A_1}{B_2 B_3} - \dfrac{A_1}{A_2 A_3} \right) f_1^+
&= - f_2^+ f_3^+, \\
\dot{f}_2^+ + \left( \dfrac{\dot{A}_2}{A_2} + \dfrac{A_2}{B_1 B_3} - \dfrac{A_2}{A_1 A_3} \right) f_2^+
&= - f_1^+ f_3^+, \\
\dot{f}_3^+ + \left( \dfrac{\dot{A}_3}{A_3} + \dfrac{A_3}{B_1 B_2} - \dfrac{A_3}{A_1 A_2} \right) f_3^+
&= - f_1^+ f_2^+.
\end{split}
\end{equation}
subject to $f_i^+=f_{i,1}^+ t + t^3 u_i(t)$, $i=1,2,3$, where $f_{i,1}^+ \in \R$ and the $u_i : (0,\infty) \rightarrow \R$ are real analytic functions such that
\begin{equation}
\begin{array}{ll}
u_1(0) =-\left( \dfrac{1}{4b_0^2} + 2 a_{2,3}+2 a_{3,3} \right) f_{1,1}^+ - f_{2,1}^+ f_{3,1}^+, \\
u_2(0) =-\left( \dfrac{1}{4b_0^2} + 2 a_{1,3}+2 a_{3,3} \right) f_{2,1}^+ - f_{1,1}^+ f_{3,1}^+, \\
u_3(0) =-\left( \dfrac{1}{4b_0^2} + 2 a_{1,3}+2 a_{2,3} \right) f_{3,1}^+ - f_{1,1}^+ f_{2,1}^+.
\end{array}
\end{equation}
\end{prop}

\begin{proof}
By Lemma \ref{lemma:extensionSU(2)^2}, we write
$$
f_i^+=f_{i,1}^+ t + t^3 u_i(t), \quad f_i^-= t^2 v_i(t),\\
$$
for some real analytic functions $u_i, v_i : (0,\infty) \rightarrow \R$, $i=1,2,3$. The new initial value problem for $X(t) =(u_1(t), u_2(t), u_3(t), v_1(t), v_2(t), v_3(t))^T$ can be written as a singular IVP:
\begin{equation}\label{eq:sivp}
\dfrac{dX}{dt}= \dfrac{M_{-1}(X)}{t} +M(t,X), \quad X(0) =X_0,
\end{equation}
where $M(t,X)$ is real analytic on the first coordinate and
\[
M_{-1}(X)= 
\begin{pmatrix}
\begin{array}{c}
-2u_1 -\left( \dfrac{1}{2b_0^2} + 4 a_{2,3}+4 a_{3,3} \right) f_{1,1}^+ - f_{2,1}^+ f_{3,1}^+ \\
-2u_2 -\left( \dfrac{1}{2b_0^2} + 4 a_{1,3}+4 a_{3,3} \right) f_{2,1}^+ - f_{1,1}^+ f_{3,1}^+ \\
-2u_3 -\left( \dfrac{1}{2b_0^2} + 4 a_{1,3}+4 a_{2,3} \right) f_{3,1}^+ - f_{1,1}^+ f_{2,1}^+ \\
-6v_1 \\
-6v_2 \\
-6v_3 \\
\end{array}
\end{pmatrix}.
\]
To show existence and uniqueness of solutions, we use \cite[Theorem 4.7]{FH17}. 
%, which was originally due to Malgrange \cite{Mal74}. 
This theorem guarantees the existence and uniqueness of short-time solutions to (\ref{eq:sivp}) in a neighbourhood of the singular orbit if $M_{-1}(0)=0$ and $h \text{Id} - d_{X_0} M_{-1}$ is invertible for all integers $h \geq 1$.
We see that $d M_{-1}(X(0))=\text{diag}(-2, -2, -2, -6, -6, -6)$, so the second condition applies. The first condition implies that $v_1(0) =v_2(0)=v_3(0)=0$ and that
\begin{equation}
\begin{array}{ll}
u_1(0) =-\left( \dfrac{1}{4b_0^2} + 2 a_{2,3}+2 a_{3,3} \right) f_{1,1}^+ - \dfrac{f_{2,1}^+ f_{3,1}^+}{2}, \\
u_2(0) =-\left( \dfrac{1}{4b_0^2} + 2 a_{1,3}+2 a_{3,3} \right) f_{2,1}^+ - \dfrac{f_{1,1}^+ f_{3,1}^+}{2}, \\
u_3(0) =-\left( \dfrac{1}{4b_0^2} + 2 a_{1,3}+2 a_{2,3} \right) f_{3,1}^+ - \dfrac{f_{1,1}^+ f_{2,1}^+}{2}.
\end{array}
\end{equation}
If the previous equation holds, then the Theorem guarantees the existence of a solution on a neighbourhood of the singular orbit.
The solution of (\ref{eq:fgh}) and $f_1^-=f_2^-=f_3^-=0$ solves the IVP, so by uniqueness all $f_i^-$ must vanish.
Remains to show that all $f_i^+$ are odd. We argue that
\begin{equation}\label{eq:Fi}
F_i(t)= \dfrac{\dot{A}_i}{A_i} + \dfrac{A_i}{B_j B_k} - \dfrac{A_i}{A_j A_k}
\end{equation}
is odd, from where we deduce that $(-f_1^+(-t), -f_2^+(-t), -f_3^+(-t))$ is a solution of (\ref{eq:fgh}). We write 
$$
-f_i^+(-t) = f^+_{i,1} t + t^3 \tilde{u}_i(t),
$$
and find $\tilde{u}_i$ with $\tilde{u}_i(0)=u_i(0)$ solving (\ref{eq:fgh}), so we can smoothly extend $f_i^+$ to $t<0$ by $f_i^+(-t)= -f_i^+(t)$ also solving (\ref{eq:fgh}), giving the desired parity conditions which guarantee the smooth extension of the instantons.
The $G_2$-instantons obtained are then parameterized by three constants $f_{1,1}^+, f_{2,1}^+, f_{3,1}^+ \in \R$.
\end{proof}

\subsection{Extension on \texorpdfstring{$P_\text{id}$}{Pid}}\label{sec:Pid}

\begin{prop}
Let $M= \R^4 \times S^3$, with coclosed $G_2$-structure given by $A_1, A_2, A_3$ as in \cite[Theorem 4.9]{Alo22}, and let $S^3 \subset M$ be the singular orbit. 
There is a 3-parameter family of $\text{SU}(2)^2$-invariant $G_2$-instantons with gauge group $\text{SU}(2)$ in a neighbourhood of $S^3$ on the bundle $P_{\text{id}}$.
In the notation of Proposition \ref{prop:Anotation}, these instantons have $f_i^\pm$ solving the system of ODEs (\ref{eq:fgh+-}).
\end{prop}

\begin{proof}
By Lemma \ref{lemma:extensionSU(2)^2}, we write
$$
f_i^+=\dfrac{2}{t}+ (b_2^+ - 4a_{i,3})t + t^3 u_i, \quad f_i^-=b_0^- + t^2 v_i,
$$
for some real analytic $u_i, v_i: (0,\infty) \rightarrow \R$, $i=1,2,3$.
Note that $v_1(0)=v_2(0)=v_3(0)$.
We can write equations (\ref{eq:fgh+-}) as a system of equations for \\
$X(t) =(u_1(t), u_2(t), u_3(t), v_1(t), v_2(t), v_3(t))^T$ which takes the form of an initial value problem:
$$
\dfrac{dX}{dt}= \dfrac{M_{-3}(b_0^-, b_2^+)}{t^3} + \dfrac{M_{-1}(X)}{t} +M(t,X).
$$
We can compute
$$
M_{-3}(b_0^-, b_2^+)= \left( -4b_2^+ + (b_0^-)^2 -\dfrac{1}{b_0^2}, -4b_2^+ + (b_0^-)^2 -\dfrac{1}{b_0^2}, -4b_2^+ + (b_0^-)^2 -\dfrac{1}{b_0^2}, 0, 0, 0 \right)^T.
$$
We must require $M_{-3}(b_0^-, b_2^+)=0$. We achieve this by imposing the extra condition
$$
4b_2^+= (b_0^-)^2 -\dfrac{1}{b_0^2}.
$$
We take $b_2^+$ such that this holds.
We now want to apply \cite[Theorem 4.7]{FH17} to the IVP
$$
\dfrac{dX}{dt}= \dfrac{M_{-1}(X)}{t} +M(t,X).
$$
The eigenvalues of $d M_{-1}$ are -8, -8, -6, -2, 0, 0, so condition (ii) of \cite[Theorem 4.7]{FH17} holds.
We also need condition (i) to hold, so we impose that $M_{-1}(X(0))=0$. We have
\[
M_{-1}(X(0))= 
\begin{pmatrix}
\begin{array}{c}
-2(u_1(0) + u_2(0) + u_3(0)) +b_0^-(v_2(0) +v_3(0)) \\
-2(u_1(0) + u_2(0) + u_3(0)) +b_0^-(v_1(0) +v_3(0)) \\
-2(u_1(0) + u_2(0) + u_3(0)) +b_0^-(v_1(0) +v_2(0)) \\
-6v_1(0) +2v_2(0) +2v_3(0) \\
+2v_1(0) -6v_2(0) +2v_3(0) \\
+2v_1(0) +2v_2(0) -6v_3(0) \\
\end{array}
\end{pmatrix}
+K,
\]
where $K$ is a constant. The computation of $K$ is more involved that in previous situations, as we need to consider Taylor expansions up to a higher order.
We consider the Taylor expansion of $A_i(t)$ up to order 5 of equation (\ref{eq:tayexpO(5)})
and then use it to compute 
$$
\dfrac{\dot{A}_1}{A_1} + \dfrac{A_1}{B_2 B_3} - \dfrac{A_1}{A_2 A_3} 
=-\dfrac{1}{t}+X_{1,1} t +X_{1,3} t^3 + O(t^5),
$$
where
$$
X_{1,1}=4 a_{2,3} +4 a_{3,3} + \dfrac{1}{2b_0^2},
$$
and
$$
X_{1,3}= \dfrac{a_{1,3}}{b_0} -\dfrac{b_2}{b_0^3} + 4(a_{1,5} + a_{2,5} +a_{3,5}) -8(a_{1,3}^2 +a_{2,3}^2 +a_{3,3}^2) +8( -a_{2,3} a_{3,3} + a_{1,3} a_{3,3} +a_{1,3} a_{2,3}).
$$
We have similar expressions for the other permutations of $\{ 1,2,3 \}$.
Recall that both $b_2, b_2^+$ can be written in terms of the data $b_0, a_{i,3}, a_{i,5}$: in particular, $b_2=1/8b_0 -b_0(a_{1,3} +a_{2,3} +a_{3,3})$ (see \cite[Lemma 4.7]{Alo22}).
The expression for $K$ that we obtain is the following:
\[
K= 
\begin{pmatrix}
\begin{array}{c}
-(b_2^+)^2 - \dfrac{b_2^+}{2b_0^2} +\dfrac{2b_2}{b_0^3} -8 (a_{1,5} + a_{2,5} +a_{3,5}) +16 (a_{1,3}^2 +a_{2,3}^2 +a_{3,3}^2) \\
-(b_2^+)^2 - \dfrac{b_2^+}{2b_0^2} +\dfrac{2b_2}{b_0^3} -8 (a_{1,5} + a_{2,5} +a_{3,5}) +16 (a_{1,3}^2 +a_{2,3}^2 +a_{3,3}^2) \\
-(b_2^+)^2 - \dfrac{b_2^+}{2b_0^2} +\dfrac{2b_2}{b_0^3} -8 (a_{1,5} + a_{2,5} +a_{3,5}) +16 (a_{1,3}^2 +a_{2,3}^2 +a_{3,3}^2) \\
2b_0^- b_2^+ - \dfrac{2b_2 b_0^-}{b_0} \\
2b_0^- b_2^+ - \dfrac{2b_2 b_0^-}{b_0} \\
2b_0^- b_2^+ - \dfrac{2b_2 b_0^-}{b_0} \\
\end{array}
\end{pmatrix}
.
\]
The solutions $X(0) =(u_1(0), u_2(0), u_3(0), v_1(0), v_2(0), v_3(0))^T$ to the non-homogeneous system of equations $M_{-1}(X(0))=0$ are given by
\begin{equation}\label{eq:v(0)}
v_1(0)=v_2(0)=v_3(0)=b_0^- b_2^+ - \dfrac{b_2 b_0^-}{b_0},
\end{equation}
and
\begin{equation}\label{eq:u(0)}
\begin{array}{ll}
u_1(0) + u_2(0) + u_3(0) 
=&-\dfrac{(b_2^+)^2}{2} - \dfrac{b_2^+}{4b_0^2} +\dfrac{b_2}{b_0^3} +(b_0^-)^2 b_2^+ - \dfrac{b_2 (b_0^-)^2}{b_0} \\
&-4 (a_{1,5} + a_{2,5} +a_{3,5}) +8 (a_{1,3}^2 +a_{2,3}^2 +a_{3,3}^2) .
\end{array}
\end{equation}
Hence, we can fix  any $u_2(0), u_3(0) \in \R$ and then $u_1(0)$ will be uniquely determined by equation (\ref{eq:u(0)}), so the solutions lie in a 2-parameter family.
Therefore, for each $b_0^-, u_2(0), u_3(0) \in \R$ there is a unique solution $X(t)$ giving a $G_2$-instanton. 
To guarantee the smoothness of the instantons found, it remains to check the parity of $f_i^\pm$ at $t=0$, i.e.\ that $f_i^+$ are odd and $f_i^-$ are even.
Again, we argue that $F_i(t)$ from (\ref{eq:Fi}) and
$$
G_i(t)= \dfrac{\dot{A}_i}{A_i} + \dfrac{A_i}{B_j B_k} - \dfrac{A_i}{A_j A_k}
$$
are odd, from where we deduce that \\
$(-f_1^+(-t), -f_2^+(-t), -f_3^+(-t), f_1^-(-t), f_2^-(-t), f_3^-(-t))$ is a solution of (\ref{eq:fgh+-}). We write 
$$
-f_i^+(-t)=\dfrac{2}{t}+ (b_2^+ - 4a_{i,3})t + t^3 \tilde{u}_i(t), \quad f_i^-(-t)=b_0^- +b_2^- t^2 \tilde{v}_i(t),
$$
and find $\tilde{u}_i, \tilde{v}_i$ with $\tilde{v}_i(0)=u_i(0)$, $\tilde{v}_i(0)=v_i(0)$, $i=1,2,3$, solving (\ref{eq:fgh+-}), so we can smoothly extend $f_i^+$ to $t<0$ by $f_i^+(-t)=-f_i^+(t)$, and $f_i^-$ to $t<0$ by $f_i^-(-t)=f_i^-(t)$ also solving (\ref{eq:fgh+-}), giving the desired parity conditions which guarantee the smooth extension of the instantons.
We have found three parameters $b_0^-, u_2(0), u_3(0) \in \R$ giving a smooth $G_2$-instanton in a neighbourhood of $S^3$.
\end{proof}

\begin{cor}
Let $M= \R^4 \times S^3$, with coclosed $G_2$-structure given by $A_1, A_2, A_3$ as in \cite[Theorem 4.9]{Alo22}, and let $S^3 \subset M$ be the singular orbit. 
Suppose there is an extra $\text{U}(1)$ symmetry, meaning that $A_2=A_3$.
Then there is a 2-parameter family of $\left( \text{SU}(2)^2 \times \text{U}(1) \right)$-invariant $G_2$-instantons with gauge group $\text{SU}(2)$ in a neighbourhood of $S^3$ on the bundle $P_{\text{id}}$.
\end{cor}
\begin{proof}
By \cite[Proposition 8]{Lotay17}, the extra symmetry means that $f_2^\pm=f_3^\pm$. The result follows by having $u_2(0)=u_3(0)$ in the previous theorem.
\end{proof}

\begin{remark}
There is a mistake in \cite[Proposition 8]{Lotay17}, which claims that in the case of $\R^4 \times S^3$ and when the $G_2$-structure is also torsion-free, there is only a 1-parameter family (instead of 2-parameter family) of $\left( \text{SU}(2)^2 \times \text{U}(1) \right)$-invariant $G_2$-instantons with gauge
group $\text{SU}(2)$ in a neighbourhood of the singular orbit on the bundle $P_\text{id}$.
\end{remark}

We leave the study of the extension of these families of instantons away from a singular orbit to a future work.

\renewcommand{\refname}{Bibliography}
\bibliographystyle{alpha}
\bibliography{biblio}

\begin{thebibliography}{dlOLM21}

\bibitem[Alo22]{Alo22}
I.~Alonso.
\newblock {Coclosed $G_2$-structures on $\text{SU}(2)^2$-invariant cohomogeneity one manifolds}, 2022.
\newblock arXiv:2209.02761.

\bibitem[BB82]{Berand1982}
L.~Bérard-Bergery.
\newblock {Sur de nouvelles variétés riemanniennes d’Einstein}.
\newblock {\em Inst. Elie Cartan. Univ. Nancy}, 6:1--60, 1982.

\bibitem[BGGG01]{Brandhuber01}
A.~Brandhuber, J.~Gomis, S.~Gubser, and S.~Gukov.
\newblock {Gauge theory at large N and new $G_2$ holonomy metrics}.
\newblock {\em J.Geom.Phys.}, 44:179--204, 2001.

\bibitem[Bog13]{Bogoyavlenskaya13}
O.~A. Bogoyavlenskaya.
\newblock {On a new family of complete $G_2$-holonomy Riemannian metrics on $S^3 \times \mathbb{R}^4$}.
\newblock {\em Sib. Math. J.}, 54(3):431–440, 2013.
\newblock Translation of Sibirsk. Mat. Zh. 54(3):551–562, 2013.

\bibitem[BS89]{BS89}
R.~L. Bryant and S.~M. Salamon.
\newblock {On the construction of some complete metrics with exceptional holonomy}.
\newblock {\em Duke Math. J.}, 58(3):829–850, 1989.

\bibitem[CHNP15]{CHNP15}
A.~Corti, M.~Haskins, J.~Nordström, and T.~Pacini.
\newblock {$G_2$-manifolds and associative submanifolds via semi-Fano 3-folds}.
\newblock {\em Duke Math. J.}, 164(10):1971–2092, 2015.

\bibitem[Cla14]{Clarke14}
A.~Clarke.
\newblock {Instantons on the exceptional holonomy manifolds of Bryant and Salamon}.
\newblock {\em J. Geom. Phys.}, 82:84–97, 2014.

\bibitem[dlOLM21]{dlOLM21}
X.~de~la Ossa, M.~Larfors, and M.~Magill.
\newblock {Almost contact structures on manifolds with a $G_2$ structure}, 2021.
\newblock arXiv:2101.12605.

\bibitem[DS11]{DS11}
S.~Donaldson and E.~Segal.
\newblock {Gauge theory in higher dimensions, II}.
\newblock {\em Surveys in differential geometry. Volume XVI. Geometry of special holonomy and related topics}, page 1–41, 2011.

\bibitem[DT98]{DT98}
S.~Donaldson and R.~Thomas.
\newblock {Gauge theory in higher dimensions}.
\newblock {\em The geometric universe (Oxford, 1996)}, page 31–47, 1998.

\bibitem[EW00]{EW00}
J.-H. Eschenburg and M.~Y. Wang.
\newblock {The initial value problem for cohomogeneity one Einstein metrics}.
\newblock {\em J. Geom. Anal.}, 10(1):109–137, 2000.

\bibitem[EW15]{SEW15}
H.~N.~Sá Earp and T.~Walpuski.
\newblock {$G_2$-instantons over twisted connected sums}.
\newblock {\em Geom. Topol.}, 19(3):1263–1285, 2015.

\bibitem[FH17]{FH17}
L.~Foscolo and M.~Haskins.
\newblock {New $G_2$-holonomy cones and exotic nearly K\"ahler structures on $S^6$ and $S^3 \times S^3$}.
\newblock {\em Ann. of Math. (2)}, 185(1):59–130, 2017.

\bibitem[FHN21]{FHN21}
L.~Foscolo, M.~Haskins, and J.~Nordström.
\newblock {Infinitely many new families of complete cohomogeneity one $G_2$-manifolds: $G_2$ analogues of the Taub--NUT and Eguchi--Hanson spaces}.
\newblock {\em J. Eur. Math. Soc.}, 23(7):2153–2220, 2021.

\bibitem[Hay17]{Haydys17}
A.~Haydys.
\newblock {$G_2$ instantons and the Seiberg-–Witten monopoles}, 2017.
\newblock arXiv:1703.06329.

\bibitem[JK21]{JK21}
D.~Joyce and S.~Karigiannis.
\newblock {A new construction of compact torsion-free $G_2$-manifolds by gluing families of Eguchi--Hanson spaces}.
\newblock {\em J. Differ. Geom.}, 117(2):255–343, 2021.

\bibitem[Joy96]{Joyce96}
D.~Joyce.
\newblock {Compact Riemannian 7-manifolds with holonomy $G_2$. I, II}.
\newblock {\em J. Differ. Geom.}, 43(2):291–328, 329–375, 1996.

\bibitem[Kov03]{Kovalev03}
A.~Kovalev.
\newblock {Twisted connected sums and special Riemannian holonomy}.
\newblock {\em J. Reine Angew. Math.}, 565:125–160, 2003.

\bibitem[LO18]{Lotay17}
J.~Lotay and G.~Oliveira.
\newblock {$SU(2)^2$-invariant $G_2$-instantons}.
\newblock {\em Math. Ann.}, 371:961--1011, 2018.

\bibitem[Mal74]{Mal74}
B.~Malgrange.
\newblock {Sur les points singuliers des équations différentielles}.
\newblock {\em Enseignement Math.}, 20(2):147–176, 1974.

\bibitem[MNE21]{MNSE21}
G.~Menet, J.~Nordström, and H.~N.~S{\'{a}} Earp.
\newblock {Construction of $G_2$-instantons via twisted connected sums}.
\newblock {\em Mathematical Research Letters}, 28(2):471--509, 2021.

\bibitem[MNT22]{MNT22}
K.~Matthies, J.~Nordström, and M.~Turner.
\newblock {$\text{SU}(2)^2 \times \text{U}(1)$-invariant $G_2$-instantons on the AC limit of the $\mathbb{C}^7$ family}, 2022.
\newblock arXiv:1812.11867.

\bibitem[MS13]{MS13}
T.~B. Madsen and S.~Salamon.
\newblock {Half-flat structures on $S^3 \times S^3$}.
\newblock {\em Ann. Glob. Anal. Geom.}, 44(4):369–390, 2013.

\bibitem[Oli14]{Oliveira14}
G.~Oliveira.
\newblock {Monopoles on the Bryant–Salamon $G_2$-manifolds}.
\newblock {\em Journal of Geometry and Physics}, 86(1):599--632, 2014.

\bibitem[Pla22]{Platt22}
D.~Platt.
\newblock {$G_2$-instantons on resolutions of $G_2$-orbifolds}, 2022.
\newblock arXiv:2208.10903.

\bibitem[SH10]{Schulte10}
F.~Schulte-Hengesbach.
\newblock {Half-flat structures on Lie groups}, 2010.
\newblock Ph.D. thesis.

\bibitem[ST23]{SteinTurner}
J.~Stein and M.~Turner.
\newblock {$G_2$-instantons on the spinor bundle of the 3-sphere}, 2023.
\newblock arXiv:2303.07967.

\bibitem[Wal13]{Walpuski13}
T.~Walpuski.
\newblock {$G_2$-instantons on generalised Kummer constructions}.
\newblock {\em Geom. Topol.}, 17(4):2345–2388, 2013.

\bibitem[Wal16]{Walpuski16}
T.~Walpuski.
\newblock {$G_2$-instantons over twisted connected sums: an example}.
\newblock {\em Math. Res. Lett.}, 23(2):529–544, 2016.

\bibitem[Wan58]{Wang58}
H.-C. Wang.
\newblock {On invariant connections over a principal fibre bundle}.
\newblock {\em Nagoya Mathematical Journal}, 13:1–19, 1958.

\end{thebibliography}

\end{document}